\documentclass[11pt]{amsart}

\usepackage{etex}
\usepackage{amsmath, amssymb}
\usepackage{array}
\usepackage[frame,cmtip,arrow,matrix,line,graph,curve]{xy}
\usepackage{graphpap, color, paralist, pstricks}
\usepackage[mathscr]{eucal}
\usepackage[pdftex]{graphicx}
\usepackage[pdftex,colorlinks,backref=page,citecolor=blue]{hyperref}
\usepackage{tikz}

\newtheorem{theorem}{Theorem}[section]
\newtheorem{proposition}[theorem]{Proposition}
\newtheorem{corollary}[theorem]{Corollary}
\newtheorem{lemma}[theorem]{Lemma}

\theoremstyle{definition}
\newtheorem{definition}[theorem]{Definition}
\newtheorem{example}[theorem]{Example}

\newtheorem{question}[theorem]{Question}
\newtheorem{remark}[theorem]{Remark}

\numberwithin{equation}{section}

\newcommand{\PP}{\mathbb{P}}
\newcommand{\QQ}{\mathbb{Q}}
\newcommand{\CC}{\mathbb{C}}

\newcommand{\ZZ}{\mathbb{Z}}

\newcommand{\cO}{\mathcal{O} }

\newcommand{\cH}{\mathcal{H} }

\newcommand{\cQ}{\mathcal{Q} }
\newcommand{\cS}{\mathcal{S} }

\newcommand{\rA}{\mathrm{A} }
\newcommand{\rH}{\mathrm{H} }
\newcommand{\rM}{\mathrm{M} }
\newcommand{\rN}{\mathrm{N} }

\newcommand{\rU}{\mathrm{U} }

\newcommand{\bB}{\mathbf{B} }
\newcommand{\bC}{\mathbf{C} }
\newcommand{\bE}{\mathbf{E} }
\newcommand{\bG}{\mathbf{G} }
\newcommand{\bH}{\mathbf{H} }
\newcommand{\bK}{\mathbf{K} }
\newcommand{\bL}{\mathbf{L} }
\newcommand{\bM}{\mathbf{M} }
\newcommand{\bR}{\mathbf{R} }
\newcommand{\bU}{\mathbf{U} }
\newcommand{\bX}{\mathbf{X} }
\newcommand{\bY}{\mathbf{Y} }
\newcommand{\bZ}{\mathbf{Z} }

\newcommand{\spec}{\mathrm{Spec}\;}
\newcommand{\proj}{\mathrm{Proj}\;}
\newcommand{\im}{\mathrm{im}}

\def\Sym{\mathrm{Sym} }
\def\Hom{\mathrm{Hom} }
\def\Ext{\mathrm{Ext} }

\def\Gr{\mathrm{Gr} }

\def\MzzGrt{\overline{\rM}_{0,0}(\Gr(n-1, V), 2)}
\def\SL{\mathrm{SL}}

\def\git{/\!/ }

\def\Eff{\mathrm{Eff} }
\def\Hilb{\mathrm{Hilb} }
\def\Bl{\mathrm{Bl} }

\newcommand{\ses}[3]{0\rightarrow{#1}\rightarrow{#2}\rightarrow{#3}\rightarrow0}

\makeatletter
\providecommand{\leftsquigarrow}{%
  \mathrel{\mathpalette\reflect@squig\relax}%
}
\newcommand{\reflect@squig}[2]{%
  \reflectbox{$\m@th#1\rightsquigarrow$}%
}
\makeatother

\hypersetup{%
pdftitle={Mori program for the moduli space of conics in Grassmannian},%
pdfauthor={Kiryong Chung, Han-Bom Moon},%
pdfkeywords={moduli space, rational curves, Grassmannian, birational geometry},%
citecolor=blue,%
linkcolor=blue,%
}

\keywords{moduli space, rational curves, Grassmannian, birational geometry}
\subjclass[2010]{14D22, 14F42, 14E15.}

\begin{document}

\title{Mori's program for the moduli space of conics in Grassmannian}
\thanks{This research was supported by Kyungpook National University Research Fund, 2015.}

\author{Kiryong Chung}
\address{Department of Mathematics Education, Kyungpook National University, 80 Daehakro, Bukgu, Daegu 41566, Korea}
\email{krchung@knu.ac.kr}

\author{Han-Bom Moon}
\address{Department of Mathematics, Fordham University, Bronx, NY 10458}
\email{hmoon8@fordham.edu}

\begin{abstract}
We complete Mori's program for Kontsevich's moduli space of degree 2 stable maps to Grassmannian of lines. We describe all birational models in terms of moduli spaces (of curves and sheaves), incidence varieties, and Kirwan's partial desingularization. 
\end{abstract}

\maketitle


\section{Introduction and results}\label{sec:intro}

\subsection{Rational curves in Grassmannian of lines}\label{ssec:curvesinGrassmannian}

The space of rational curves in $\Gr(2,n)$ and its compactifications has been studied in various context. In the study of Fano manifolds, the space of lines in $\Gr(2,n)$ has been one of the main tools to study the geometry of the linear or quadratic sections of $\Gr(2,n)$ (\cite{Pro94}). In fact, the codimension two linear section of $\Gr(2, 5)$ is the answer for Hirzebruch's question in dimension 4: Classify all minimal compactifications of $\CC^4$.

On the other hand, a compactified moduli space of conics in $\Gr(2,n)$ for $n=5$ has been an essential ingredient in the construction of a new compact hyperk\"ahler manifold. For example, in \cite{IM11}, by following the general construction of a symplectic two-form on the moduli space of sheaves or rational curves (\cite{dJS04, KM09}), the authors proved that a certain contraction of the Hilbert scheme of conics in $\Gr(2,5)\cap H_1\cap H_2$ where $H_{d}$ is a hypersurface of degree $d$, is a hyperk\"ahler manifold discovered by O'Grady in \cite{OG06}. 

In the study of homological mirror symmetry, it is important to present a pair of Calabi-Yau threefolds which are derived equivalent but not birationally equivalent. Only few examples of such pairs have been known. In \cite{HT13}, the authors provided new such a pair by using the double cover (the so-called \emph{double symmetroid}) of the determinantal symmetroid in the space of quadrics $\PP(\Sym^2{\CC^{5}}^{*})$. One of the main steps of the construction is to find an explicit birational relation between the double symmetroid and the Hilbert scheme of conics in the Grassmannian $\Gr(3,5) \cong \Gr(2,5)$ of planes. This relation has been established in \cite{HT15} in a broader setting, namely, for the space of quadrics in $\PP(\Sym^{2}{\CC^{n+1}}^{*})$ and the Hilbert scheme of conics in $\Gr(n-1, n+1) \cong \Gr(2, n+1)$. 

\subsection{Main results}

The main result of this paper is the completion of the projective birational geometry of the space of conics in $\Gr(n-1, n+1)$ in the viewpoint of \emph{Mori's program}. Mori's program, or the log minimal model program for a projective moduli space $M$ aims the classification of all rational contractions of $M$. If $M$ is a Mori dream space, then for each effective divisor $D$, one can construct a projective model 
\[
	M(D) := \proj \bigoplus_{m \ge 0}\rH^{0}(M, \cO(mD))
\]
and a rational contraction $M \dashrightarrow M(D)$. Provided by being a Mori dream space, there are only finitely many projective models. 

For the moduli space $\overline{\rM}_{0,0}(X, d)$ of stable maps, which is a compactification of rational curves in a projective variety $X$, Mori's program has been studied in \cite{Che08b, CC11}. When $X = \Gr(k, V)$, the Grassmannian of subspaces in $V$ and $d = 2, 3$, the stable base locus decomposition was obtained by Chen and Coskun in \cite{CC10}, as a first step toward Mori's program. 

In this paper, we complete Mori's program for $\overline{\rM}_{0,0}(\Gr(n-1, V), 2)$. Furthermore, we describe all birational models in terms of moduli spaces, incidence varieties, and partial desingularizations (\cite{Kir85}). 

Let $V$ be a vector space of dimension $n + 1 \ge 5$. For the precise definition of the divisors in the statement below, see Definition \ref{def:divisorclasses}.

\begin{theorem}\label{thm:mainintro}
For an effective divisor $D$ on $\bM := \MzzGrt$,  
\begin{enumerate}
\item If $D = aH_{\sigma_{1, 1}} + bH_{\sigma_{2}} + cT$ for $a, b, c > 0$, then $\bM(D) \cong \bM$. 
\item If $D = aH_{\sigma_{1, 1}} + bH_{\sigma_{2}}$ for $a, b > 0$, then $\bM(D) \cong \bC := \mathrm{Chow}_{1, d}(\Gr(n-1, V))^{\nu}$, the normalization of the main component of Chow variety. 
\item If $D = aH_{\sigma_{1, 1}} + bH_{\sigma_{2}} + cP$ for $a, b, c > 0$, then $\bM(D) \cong \bH := \Hilb^{2m+1}(\Gr(n-1, V))$. 
\item If $D = aT + b\Delta$ for $a > 0$ and $b \ge 0$, then $\bM(D) \cong \bU := \overline{\rU}_{0,0}(\Gr(n-1, V), 2)$, the normalization of the closure of the image of $\bM$ in $\PP(\wedge^{n-1}V^{*}\otimes \mathfrak{sl}_{2})\git \SL_{2}$.
\item If $D = aH_{\sigma_{2}} + bD_{\deg} + c\Delta$ for $a > 0$ and $b, c \ge 0$, then $\bM(D) \cong \bK := \PP(V^{*}\otimes \mathfrak{gl}_{2})\git \SL_{2} \times \SL_{2}$, which is isormorphic to a connected component of the moduli space $\rM_{\PP V}(v)$ of semistable sheaves with $v = 2\mathrm{ch}(\cO_{\PP^{n-1}})$ (Definition \ref{def:quasimap}). 
\item If $D = aH_{\sigma_{2}} + bT + c\Delta$ for $a, b > 0$ and $c \ge 0$, then $\bM(D) \cong \bX^{1}\git \SL_{2} \times \SL_{2}$, which is the first step of the partial desingularization of $\bK$ (Section \ref{ssec:partialdesing}). 
\item If $D = aH_{\sigma_{2}} + bP + cD_{\deg}$ for $a , b > 0$ and $c \ge 0$, then $\bM(D) = \widetilde{\bG}$ (Section \ref{ssec:determinantalvarieties}). 
\item If $D = aD_{\mathrm{unb}} + bP + cD_{\deg}$ for $a , b > 0$ and $c \ge 0$, then $\bM(D) = \bG := \Gr(3, \wedge^{2}\cS)$ where $\cS$ is the universal subbundle over $\Gr(4, V^{*})$. 
\item If $D = aH_{\sigma_{1, 1}} + bP + cD_{\mathrm{unb}}$ for $a, b, c > 0$, then $\bM(D) \cong  \bB := \Bl_{\mathrm{OG}(3, \wedge^{2}\cS)_{\sigma_{(2)^{*}}}}\Gr(3, \wedge^{2}\cS)$ (Definition \ref{def:B}). 
\item If $D = aH_{\sigma_{1, 1}} + bD_{\mathrm{unb}} + c\Delta$ for $a, b > 0$ and $c \ge 0$, then $\bM(D) \cong \bK_{\cS} := \PP(\cS^{*} \otimes \mathfrak{gl}_{2})\git \SL_{2} \times \SL_{2} \cong \rM_{\PP \cS}(m^{2}+3m+2)$, the relative moduli space of semistable sheaves (Definition \ref{def:relativesheaves}). 
\item If $D = aH_{\sigma_{1, 1}} + bT + c\Delta$ for $a, b > 0$ and $c \ge 0$, then $\bM(D)$ is the normalization $\bR$ of the incidence variety in $\rM_{\PP V^{*}}(m^{2}+3m+2) \times \bU$.
\item If $D = aH_{\sigma_{1,1}} + b\Delta$ for $a > 0$ and $b \ge 0$, then $\bM(D) \cong \bL$, the normalization of the closure of the locus of sheaves supported on a smooth quadric surfaces in $\rM_{\PP V^{*}}(m^{2}+3m+2)$ (Definition \ref{def:bL}). 
\item If $D = aP + bD_{\deg}$ for $a > 0$ and $b \ge 0$, then $\bM(D) \cong \overline{\bG}$, which is the normalization of the image of the envelope map $env : \bH \to \Gr(3, \wedge^{2}V^{*})$ (Definition \ref{def:envelope}). 
\item If $D = aH_{\sigma_{1,1}} + bP$ for $a, b > 0$, then $\bM(D)$ is the blow-up $\widehat{\bG}$ of $\bG$ along a subvariety isomorphic to $\mathrm{OG}(3, \wedge^{2}\cS)$. 
\item If $D = a\Delta + bD_{\deg}$ for $a, b \ge 0$, then $\bM(D)$ is a point. 
\item If $D = aD_{\mathrm{unb}} + b\Delta$ for $a > 0$ and $b \ge 0$, or $D = aD_{\mathrm{unb}} + bD_{\deg}$ for $a > 0$ and $b \ge 0$, then $\bM(D)$ is $\Gr(4, V^{*}) \cong \Gr(n-3, V)$.
\end{enumerate}
\end{theorem}

When $n = 3$, the description of birational models is simpler because of the self-dual map on $\Gr(2, 4)$. See Theorem \ref{thm:n=3case} for the statement. 

Note that there are only few examples of completed Mori's program when the complexity of the moduli space is large. Except toric varieties and moduli spaces with Picard number two (for instance \cite{Che08b, Moo14b}), the completed examples are very rare (see \cite{MY16} for such an example). Theorem \ref{thm:mainintro} provides one additional example with Picard number three. 

\subsection{Application to the motivic invariants}
 
Let us finish this section by mentioning some related works. One of birational models of $\bM$ turns out to be the moduli space of quiver representations with dimension vector $(2, 2)$ and $n+1$ arrows between them (Item (5) of Theorem \ref{thm:mainintro}). We call the moduli space of such quiver representations by the moduli space of Kronecker modules of type $(n+1;2,2)$, or simply, the \emph{Kronecker moduli space}. The Kronecker moduli space has been studied in the context of curve counting invariants (In particular, GW/Kronecker correspondence). For the detail, see \cite{Sto11}. 

The Kronecker moduli space of type $(6;2,2)$ is birational to the moduli space of semistable sheaves on $\PP^{2}$ with Hilbert polynomial $4m+2$. The birational map can be explicitly described in terms of Bridgeland wall-crossing (\cite[Section 6]{BMW14}). Combining with our analysis, we obtain the virtual Poincar\'e polynomial of $\rM_{\PP^{2}}(4m+2)$ from that of $\bM$. For a detail, see Section \ref{sec:application}.
 
\subsection{Structure of paper}

This paper is organized as the following. In Section \ref{sec:stablebaselocus}, for the reader's convenience, we recall the stable base locus decomposition of $\bM$ obtained by Chen and Coskun. Section \ref{sec:birationalmodels} introduces many birational models obtained in \cite{CC10, HT15}. In Section \ref{sec:Kronecker}, we study geometric properties of the moduli space $\bK$ of Kronecker modules, which is a key ingredient on the moduli theoretic interpretation of biratoinal models. In Section \ref{sec:pardesing} we show that the partial desingularization of $\bK$ is indeed $\bM$ and investigate the geometry of the birational contraction. After introducing some more natural models, in Section \ref{sec:Moriprogram} we prove Theorem \ref{thm:mainintro}. Finally, in the last section we compute topological invariants of some moduli spaces. 

\subsection{Notation}
 
We work on $\CC$. A projective space $\PP V$ is the space of one-dimensional \emph{subspaces} of $V$. For a partition $\lambda$, let $\Sigma_{\lambda}$ be an associated Schubert variety in $\Gr(k, V)$. Its Poincar\'e dual is denoted by $\sigma_{\lambda}$. For a partition $\lambda$, $\lambda^{*}$ is the dual partition. For a direct sum of sheaves, we will use additive notation. For instance, $2\cO_{X}$ means $\cO_{X}^{\oplus2}$. 



\section{Stable base locus decomposition}\label{sec:stablebaselocus}

In this section, we fix an integer $n \ge 3$. Let $V$ be an ($n+1$)-dimensional vector space and let $k$ be an integer such that $2 \le k \le n-1$. 

When one runs Mori's program for a given moduli space $M$, the first step is the computation of the rank of Neron-Severi vector space $\rN^{1}(M)$ and the effective cone $\Eff(M)$. For the moduli space $\overline{\rM}_{0,0}(\Gr(k, V), 2)$, $\dim \rN^{1}(\overline{\rM}_{0,0}(\Gr(k, V), 2)) = 3$ (\cite[Theorem 1]{Opr05}). Its effective cone was computed by Coskun and Starr in \cite{CS06}. To describe the result, we need to introduce several effective divisor classes on $\overline{\rM}_{0,0}(\Gr(k, V), 2)$. 

\begin{definition}\label{def:divisorclasses}
\begin{enumerate}
\item Let $\Delta$ be the locus of stable maps with singular domains.
\item Fix an $n-1-k$-dimensional subspace $W$ of $V$. Let $D_{\deg}$ be the locus of stable maps $f$ such that the projection of the smallest linear subspace in $V$ generated by the image of $f$ onto $V/W$ is a proper subspace of $V/W$. If $k = n-1$, $D_{\deg}$ is the locus of stable maps whose image is in $\Gr(n-1, V')$ for some $n$-dimensional subspace $V' \subset V$. 
\item For a stable map $f : \PP^{1} \to \Gr(k, V)$, let $0 \to E \to \cO_{\PP^{1}} \otimes V$ be the induced subbundle of rank $k$ of degree $-2$.  If $k = 2$, let $D_{\mathrm{unb}}$ be the closure of the locus of stable maps such that $E \ne 2\cO_{\PP^{1}}(-1)$. When $k > 2$, for a general stable map $f$ and its associated subbundle $0 \to E \to \cO_{\PP^{1}} \otimes V$, there is a trivial subbundle $E' := (k-2)\cO_{\PP^{1}} \subset E$ which induces an ($k-2$)-dimensional sub vector space $V_{E'} \subset V$. Let $D_{\mathrm{unb}}$ be the closure of the locus of stable maps such that $V_{E'} \cap W \ne \{0\}$ for a fixed ($n+3-k$)-dimensional subspace $W \subset V$. In other words, $D_{\mathrm{unb}} = a^{*}(\cO_{\Gr(n+3-k, V)}(1))$ for $a : \overline{\rM}_{0,0}(\Gr(k, V), 2) \dashrightarrow \Gr(n+3-k, V)$. 
\item Let $H_{\sigma_{1,1}}$ (resp. $H_{\sigma_{2}}$) be the locus of stable maps whose image in $\Gr(k, V)$ intersects a fixed codimension two Schubert variety $\Sigma_{1,1}$ (resp. $\Sigma_{2}$). 
\item Let $T$ be the locus of stable maps whose image in $\Gr(k, V)$ is tangent to a fixed hyperplane $\Sigma_{1}$. 
\item If we compose a general stable map $f : C \to \Gr(k, V)$ with the Pl\"ucker embedding $\Gr(k, V) \hookrightarrow \PP(\wedge^{k}V)$, then we obtain a conic in $\PP(\wedge^{k}V)$, which spans a two-dimensional subspace in $\PP(\wedge^{k}V)$. Thus we obtain an element in $\Gr(3, \wedge^{k}V)$. Thus there is a rational map $p : \overline{\rM}_{0,0}(\Gr(k, V), 2) \dashrightarrow \Gr(3, \wedge^{k}V)$. Let $P := p^{*}\cO_{\Gr(3, \wedge^{k}V)}(1)$. 
\end{enumerate}
\end{definition}

\begin{theorem}[\protect{\cite[Section 2]{CS06}}]
The effective cone of $\overline{\rM}_{0,0}(\Gr(k, V), 2)$ is generated by $D_{\mathrm{unb}}$, $D_{\deg}$ and $\Delta$. In particular, the effective cone is simplicial. 
\end{theorem}

The next step is the computation of the stable base locus decomposition, which is the first approximation of the Mori chamber decomposition of the effective cone. This was done by Chen and Coskun in \cite[Theorem 3.6]{CC10}. In particular, there are 8 open chambers as in Figure \ref{fig:stablebaselocusdecomposition}. Also the divisor classes of $T$, $D_{\deg}$, $D_{\mathrm{unb}}$, and $P$ are calculated in terms of $\Delta$, $H_{\sigma_{1,1}}$, and $H_{\sigma_{2}}$ in \cite[Section 4, 5]{CS06} and \cite[Lemma 3.4]{CC10}. 

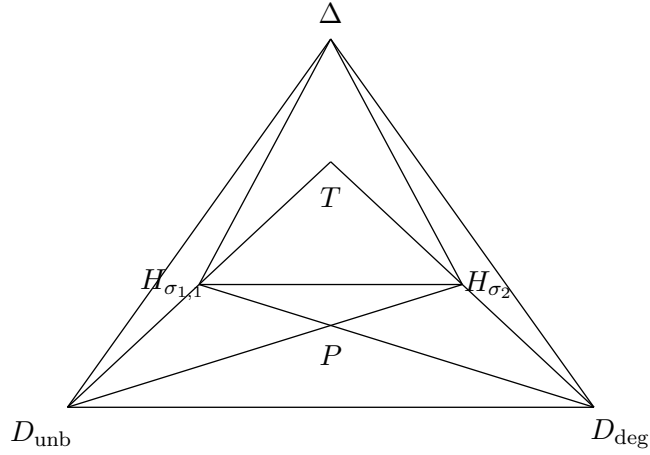
\begin{figure}[!ht]
\begin{tikzpicture}[scale=0.7]
\draw[line width = 0.5 pt] (0, 0) -- (10, 0);
\draw[line width = 0.5 pt] (0, 0) -- (5, 7);
\draw[line width = 0.5 pt] (5, 7) -- (10, 0);
\draw[line width = 0.5 pt] (0, 0) -- (5, 14/3);
\draw[line width = 0.5 pt] (10, 0) -- (5, 14/3);
\draw[line width = 0.5 pt] (2.5, 7/3) -- (7.5, 7/3);
\draw[line width = 0.5 pt] (2.5, 7/3) -- (10, 0);
\draw[line width = 0.5 pt] (7.5, 7/3) -- (0, 0);
\draw[line width = 0.5 pt] (2.5, 7/3) -- (5, 7);
\draw[line width = 0.5 pt] (7.5, 7/3) -- (5, 7);

\node (2) at (-0.5, -0.5) {$D_{\mathrm{unb}}$};
\node (3) at (10.5, -0.5) {$D_{\deg}$};
\node (4) at (5, 7.5) {$\Delta$};
\node (5) at (5, 4) {$T$};
\node (6) at (2, 7/3) {$H_{\sigma_{1, 1}}$};
\node (7) at (8, 7/3) {$H_{\sigma_{2}}$};
\node (8) at (5, 1) {$P$};
\end{tikzpicture}
\caption{Stable base locus decomposition of $\overline{\rM}_{0,0}(\Gr(k, V), 2)$}
\label{fig:stablebaselocusdecomposition}
\end{figure}

For the computation of the stable base locus decomposition, Chen and Coskun introduced many curve classes on $\overline{\rM}_{0,0}(\Gr(k, V), 2)$ in \cite[Section 3]{CC10}. Since these curves will have a prominent role in the proof of Theorem \ref{thm:mainintro}, for a reader's convenience, we leave the definition.

\begin{definition}\label{def:curveclasses}
\begin{enumerate}
\item Let $C_{1}$ (resp. $C_{2}$) be a general pencil of conics in a fixed plane of class $\Sigma_{(1,1)^{*}}$ (resp. $\Sigma_{(2)^{*}}$).
\item Let $C_{5}$ be a one-parameter family of conics in a fixed $\Sigma_{(1,1)^{*}}$ tangent to four general lines.
\item Let $C_{6}$ (resp. $C_{7}$) be a one-dimensional family of singular stable maps obtained by attaching a fixed line to the base point of a pencil of lines in a fixed $\Sigma_{(1,1)^{*}}$ (resp. $\Sigma_{(2)^{*}}$). 
\item Let $C_{8}$ be a one-parameter family of two-to-one covers of a fixed line. 
\end{enumerate}
\end{definition}

The intersection numbers of curve classes with divisors are summarized in \cite[Table 1]{CC10}.

\section{Some birational models}\label{sec:birationalmodels}

From this section, we focus on $k = n-1$ case, that is, $M = \bM = \MzzGrt$. 

After the computation of the stable base locus decomposition, the next step of Mori's program of $\bM$ is to determine a birational model $\bM(D)$ for each effective divisor $D$. Since $\bM$ is a Mori dream space (\cite[Corollary 1.2]{CC10}), there are finitely many birational models for each cone in the stable base locus decomposition. 

For $\bM$, there has been two prior results. A family of birational models is obtained in \cite{CC10} from the moduli theoretic viewpoint. On the other hand, by using multilinear algebra and incidence varieties, another family of birational models is obtained by Hosono and Takagi in \cite{HT15}. In this section, we review these birational models. 

\subsection{Moduli theoretic models}\label{ssec:modularmodels}

Since $\bM$ is a compactification of the moduli space of smooth conics in $\Gr(n-1, V)$, we obtain several natural birational models from different compactifications of moduli spaces of smooth conics. Here we review such birational models. 

\begin{definition}
\begin{enumerate}
\item Let $\bH := \Hilb^{2m+1}(\Gr(n-1, V))$ be Hilbert scheme of conics in $\Gr(n-1, V)$. 
\item Let $\bC := \mathrm{Chow}_{1, 2}(\Gr(n-1, V))^{\nu}$ be the normalization of the main component of Chow variety of dimension 1, degree 2 algebraic cycles in $\Gr(n-1, V)$. 
\end{enumerate}
\end{definition}

\begin{definition}\label{def:quasimap}
For a projective space $\PP W$, there is a divisorial contraction (\cite[Section 11]{Giv96}, \cite[Theorem 0.1]{Par07})
\[
	\overline{\rM}_{0,0}(\PP W, 2) \to 
	\PP(W^{*} \otimes \Sym^{d}\CC^{2}) \git \SL_{2}.
\]
This map is indeed the partial desingularization (\cite{Kir85}) when $d = 2$ (\cite[Theorem 4.1]{Kie07}). Let $\bU := \overline{\rU}_{0,0}(\Gr(n-1, V), 2)$ be the normalization of the image of the composition
\[
	\bM \hookrightarrow \overline{\rM}_{0,0}(\PP(\wedge^{n-1}V), 2) 
	\to \PP(\wedge^{n-1}V^{*}\otimes \mathfrak{sl}_{2}) \git \SL_{2}.
\]
\end{definition}

It is well-known that $\bH$ is smooth (\cite[Proposition 3.6]{CHK12}), and there is a diagram
\[
	\xymatrix{\bM \ar[d]\ar[rd] \ar@{<-->}[rr] && \bH\ar[ld]\\
	\bU & \bC.}
\]

\begin{definition}\label{def:envelope}
For each conic $C \in \bH = \Hilb^{2m+1}(\Gr(n-1, V)) \subset \Hilb^{2m+1}(\PP(\wedge^{n-1}V))$, the smallest linear subspace $\langle C \rangle$ of $\PP(\wedge^{n-1}V)$ containing $C$, the so-called \emph{linear envelope} of $C$, is $\PP^{2}$. Thus we have a regular morphism
\begin{equation}\label{eqn:envelope}
\begin{split}
	env : \bH \to&\; \Gr(3, \wedge^{n-1}V)\\
	C \mapsto &\; \langle C\rangle.
\end{split}
\end{equation}
\end{definition}

For any conic $C$, $\langle C \rangle \cap \Gr(n-1, V) \subset \PP(\wedge^{n-1}V)$ is either $C$ or $\langle C\rangle \cong \PP^{2}$. The second case happens precisely when $\langle C\rangle \subset \Gr(n-1, V)$. The moduli space of planes in $\Gr(n-1, V)$ has two connected components. One is the moduli space of Schubert planes $\Sigma_{(1,1)^{*}}$, which is isomorphic to a partial flag variety $\mathrm{Fl}(n-3, n, V)$. The other component is the moduli space of Schubert planes $\Sigma_{(2)^{*}}$, which is isomorphic to $\Gr(n-2, V)$. 

When $n = 3$, $env$ is a birational morphism. Indeed, $env$ is the blow-up along two disjoint orthogonal Grassmannians $\mathrm{OG}(3, \wedge^{2}V)$ which parametrize $\sigma_{(1,1)^{*}}$ (resp. $\sigma_{(2)^{*}}$) planes in $\Gr(2, V)$ (\cite[Lemma 3.9]{CC10}). We denote them by $\mathrm{OG}(3, \wedge^{2}V)_{\sigma_{(1,1)^{*}}}$ (resp. $\mathrm{OG}(3, \wedge^{2}V)_{\sigma_{(2)^{*}}}$) whenever we want to distinguish them. 

In summary, 
\begin{equation}\label{eqn:HilbblowupGr}
	\Hilb^{2m+1}(\Gr(2, V)) \cong 
	\mathrm{Bl}_{2\mathrm{OG}(3, \wedge^{2}V)}\Gr(3, \wedge^{2}V).
\end{equation}

\subsection{Models from birational geometry of determinantal varieties}\label{ssec:determinantalvarieties}

With a motivation in the context of homological projective duality, Hosono and Takagi studied birational geometry of $T_{r}$, which is a double cover of the space $S_{r}$ of rank $\le r$ quadric hypersurfaces in $\PP V$ (\cite{HT15}). More precisely, for $1 \le r \le n+1$, let $S_{r} \subset \PP (\Sym^{2}V^{*})$ be the locus that the associated quadratic form has rank $\le r$. When $r$ is even, there is a double cover $T_{r}$ of $S_{r}$ ramified along $S_{r-1}$ (\cite[Proposition 2.3]{HT15}). 

Set theoretically, $T_{4}\setminus S_{3}$, which is an \'etale double cover of $S_{4}\setminus S_{3}$, parametrizes pairs $(Q, P)$ where $Q$ is a rank 4 quadric hypersurface and $P$ is a pencil of $\PP^{n-2}$ in $Q$. Furthermore, they show that when $n \ge 3$, $T_{4}$ is birational to $\bH$ and there is a contraction diagram:
\[
	\xymatrix{&& \bH \ar[d]\\
	\bG \ar[rd] \ar@{<-->}[rr] && 
	\widetilde{\bG} \ar[ld] \ar[rd]\\
	& \overline{\bG} && T_{4},}
\]
where
\begin{enumerate}
\item $\bG := \Gr(3, \wedge^{2}\cS)$ is the Grassmannian bundle where $\cS$ is the rank 4 universal subbundle over $\Gr(4, V^{*})$;
\item $\overline{\bG}$ is the normalization of the image of $env : \bH \to \Gr(3, \wedge^{n-1}V)$ (see Definition \ref{def:envelope});
\item $\widetilde{\bG}$ is the $D$-flip of $\bG$ over $\overline{\bG}$;
\item $\widetilde{\bG}$ is a divisorial contraction of $\bH$, which contracts the curve class $C_{1}$, and hence $D_{\deg}$.
\end{enumerate}

Furthermore, the following properties are studied in \cite[Proposition 4.22, 4.11, and 4.5]{HT15}.
\begin{enumerate}
\item The map $\widetilde{\bG} \to T_{4}$ is a divisorial contraction which contracts the image of $\Delta$;
\item The normalization $\overline{\bG} \to \im\; env$ is bijective;
\item If $n = 3$, then $\bG \cong \overline{\bG} \cong \Gr(3, \wedge^{2}V)$. If $n > 3$, then $\bG \to \overline{\bG}$ is a small contraction whose exceptional fibers are $\PP^{n-3}$;
\item The map $\widetilde{\bG} \to \overline{\bG}$ is a contraction of the image of the locus of conics in a $\Sigma_{(2)^{*}}$-plane. So for an exceptional point, its fiber is $\PP^{5}$. If $n = 3$, it is a blow-up, but if $n > 3$, this is a small contraction. 
\end{enumerate}


\section{Moduli space of Kronecker modules}\label{sec:Kronecker}

In this paper, the moduli space of Kronecker modules has a central role to connect moduli spaces of sheaves and that of rational curves. In this section, we review its definition and basic properties. 

\subsection{Definitions and GIT stability}\label{ssec:GITstability}

Fix two positive integers $a, b$ and let $W$ be a vector space. A \emph{Kronecker $W$-module} is a quiver representation of an $n$-Kronecker quiver 
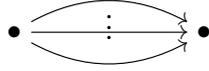
\[
	\xymatrix{\bullet\ar@/^1.0pc/[rr] \ar@/_1.0pc/[rr]\ar[rr]
	&{\vdots}&\bullet\\}
\]
with a dimension vector $(a, b)$. Two Kronecker $W$-modules $\phi = (\phi_{i})$ and $\psi = (\phi_{i})$ are equivalent if there are $A \in \SL_{a}$ and $B \in \SL_{b}$ such that $\phi = B \circ \psi \circ A$. We may regard the GIT quotient 
\[
	\PP\Hom(W\otimes \CC^{a}, \CC^{b})\git \SL_{a} \times \SL_{b}
\]
as a moduli space of \emph{semistable} Kronecker $W$-modules. The GIT stability was obtained by Drezet (\cite[Proposition 15]{Dre87}). 

\begin{theorem} 
A closed point $M \in \PP \Hom(W \otimes \CC^{a}, \CC^{b})$ is (semi)stable with respect to $\SL_{a} \times \SL_{b}$-action if and only if for every nonzero proper subspace $V_{1} \subset \CC^{a}$ and $V_{2} \subset \CC^{b}$ such that $M(W \otimes V_{1}) \subset V_{2}$, 
\[
	\frac{\dim V_{2}}{\dim V_{1}} \;(\ge) > \frac{b}{a}.
\]
\end{theorem}

From now, we restrict ourselves to a special case that $a = b = 2$. 

\begin{corollary}
A closed point $M \in \PP \Hom(W \otimes \CC^{2}, \CC^{2}) \cong \PP(W^{*}\otimes \mathfrak{gl}_{2})$ is (semi)stable with respect to $\SL_{2} \times \SL_{2}$ if and only if for every one-dimensional subspace $V_{1} \subset \CC^{2}$, $\dim \;\im \; M(W \otimes V_{1}) (\ge) > 1$. 
\end{corollary}

From now, let $V$ be an ($n+1$)-dimensional vector space, as before. We may describe $M \in \bX := \PP\Hom(V\otimes \CC^{2}, \CC^{2}) = \PP(V^{*}\otimes \mathfrak{gl}_{2})$ as a $2 \times 2$ matrix of linear polynomials with $n+1$ variables $x_{0}, \cdots, x_{n}$. Then $M$ is semistable if and only if even after performing row/column operations, there is no zero row or column. $M$ is stable if and only if $M$ has no zero entry. In summary, we have the following result. 

\begin{lemma}\label{lem:stability}
Let $M \in \bX := \PP(V^{*}\otimes \mathfrak{gl}_{2})$. 
\begin{enumerate}
\item If $M$ is unstable, then $M$ is equivalent to 
\[
	\left[\begin{array}{cc}g&h\\0&0\end{array}\right]
	\mbox{ or }
	\left[\begin{array}{cc}g&0\\h&0\end{array}\right]
\]
for some $g, h \in V^{*}$. 
\item If $M$ is strictly semistable, then $M$ is equivalent to 
\[
	\left[\begin{array}{cc}g&k\\0&h\end{array}\right]
\]
for some $g, h \in V^{*}\setminus \{0\}$ and $k \in V^{*}$. 
\item If $M$ is strictly semistable and has a closed orbit in the semistable locus, then $k = 0$, so $M$ is equivalent to 
\[
	\left[\begin{array}{cc}g&0\\0&h\end{array}\right]
\]
for some $g, h \in V^{*}\setminus \{0\}$. If $g$ is proportional to $h$, then $\mathrm{Stab}\; M \cong \SL_{2} \ltimes \ZZ_{2}$. If $g$ is not proportional to $h$, then $\mathrm{Stab}\;M \cong \CC^{*} \ltimes \ZZ_{2}$.
\end{enumerate}
\end{lemma}

\begin{remark}
The description of stabilizers is different from that in \cite[Lemma 6.4]{CM16}, because in this paper we are taking the $\SL_{2} \times \SL_{2}$ quotient instead of the $\mathrm{PGL}_{2} \times \mathrm{PGL}_{2}$ quotient.
\end{remark}

\subsection{Moduli space of Kronecker modules as a moduli space of semistable sheaves}\label{ssec:Kroneckerasmoduliofsheaves}

The moduli spaces of Kronecker modules can be understood as moduli spaces of semistable sheaves. Some explicit examples can be found in \cite{Dre87} and \cite[Section 3]{LP93b}. In this section we investigate a generalization toward moduli spaces of sheaves on higher dimensional projective spaces. 

The following lemma is a direct generalization of \cite[Lemma 5.2]{CM16}.

\begin{lemma}\label{lem:stabilitycomparison}
Let $n \ge 2$. Let $F \in \mathsf{Coh}(\PP^{n})$ has a resolution 
\begin{equation}\label{eqn:resolution}
	0 \to 2\cO_{\PP^{n}}(-1) \stackrel{M}{\to} 2\cO_{\PP^{n}} \to F
	\to 0
\end{equation}
such that $M$ is a semistable Kronecker module. Then $F$ is isomorphic to either \begin{enumerate}
\item $F = I_{\PP^{n-2}, Q}(1)$ for some quadric hypersurface $Q$ of rank 3 or 4;
\item $F$ is an extension of $\cO_{H}$ by $\cO_{H'}$ for two hyperplanes $H, H'$.
\end{enumerate}
In particular, $F$ is semistable. Furthermore, in the case of (1), $F$ is stable. 
\end{lemma}

\begin{proof}
By composing $M$ with an injective morphism $\cO_{\PP^{n}}(-1) \to 2\cO_{\PP^{n}}(-1)$, we obtain an injective morphism $0 \to \cO_{\PP^{n}}(-1) \to 2\cO_{\PP^{n}}$ whose cokernel is isomorphic to either $I_{L, \PP^{n}}(1)$ for a linear subspace $L$ of dimension $n-2$, or $\cO_{H}\oplus \cO_{\PP^{n}}$. 

\textsf{Case 1.} Suppose that the cokernel is isomorphic to $I_{L, \PP^{n}}(1)$. 

We have a commutative diagram
\[
	\xymatrix{0 \ar[r]& \cO_{\PP^{n}}(-1) \ar[r] \ar[d] 
	& 2\cO_{\PP^{n}} \ar[r] \ar@{=}[d] & I_{L, \PP^{n}}(1) \ar[d] \ar[r]
	& 0\\
	0 \ar[r] & 2\cO_{\PP^{n}}(-1) \ar[r] & 2\cO_{\PP^{n}} \ar[r]
	& F \ar[r] & 0.}
\]
Apply the snake lemma, then we obtain 
\[
	\ses{\cO_{\PP^{n}}(-1)}{I_{L, \PP^{n}}(1)}{F}.
\]
From the sequence $\ses{I_{L, \PP^{n}}(1)}{\cO_{\PP^{n}}(1)}{\cO_{L}(1)}$ and the snake lemma again, we obtain $F \cong I_{L, Q}(1)$ for some quadric hypersurface $Q$. Here $Q$ is the support of $\cO_{\PP^{n}}(-1) \to \cO_{\PP^{n}}(1)$. Since the defining equation of $Q$ is in $\rH^{0}(I_{L}(2))$, it has rank at most 4. 

If $Q$ has rank 3 or 4, it is irreducible and reduced. Then every subsheaf of $I_{L, Q}(1)$ is of the form $I_{Z, Q}(1)$ for some subscheme $L \subset Z \subset Q$. If $\dim Z = \dim L$, then clearly $p(I_{Z, Q}(1)) < p(I_{L, Q}(1))$ where $p$ is the reduced Hilbert polynomial. If $\dim Z = \dim Q$, since $Q$ is irreducible and reduced, $Z = Q$ and $I_{Z, Q}(1) = 0$. Thus $I_{L, Q}(1)$ is stable. 

If $Q$ has rank $\le 2$, then $Q = H \cup H'$ or $2H$ for two hyperplanes $H, H'$. Suppose that $L \subset H'$. From $\ses{I_{H', H\cup H'}(1)}{I_{L, H\cup H'}(1)}{I_{L, H'}(1)}$, we obtain 
\[
	\ses{\cO_{H'}}{I_{L, H\cup H'}(1)}{\cO_{H}}.
\]
Thus $F = I_{L, H \cup H'}$ is semistable. Furthermore, since $p(\cO_{H'}) = p(I_{L, H\cup H'}(1))$, $F$ is strictly semistable. $Q = 2H$ case is similar. 

\textsf{Case 2.} Assume that the cokernel is $\cO_{H}\oplus \cO_{\PP^{n}}$.

In this case, it is straightforward to see that $M$ is represented by a matrix in item (2) or (3) in Lemma \ref{lem:stability}. Then $F = I_{L, H\cup H'}(1)$ (in the case of (2)) or $F = \cO_{H}\oplus \cO_{H'}$ (in the case of (3)) and $F$ fits in an exact sequence $\ses{\cO_{H'}}{F}{\cO_{H}}$.
\end{proof}

\begin{proposition}
Let $\rM_{\PP V}(v)$ be the moduli space of sesmistable pure sheaves $F$ with $v:= ch(F) = 2ch(\cO_{\PP^{n-1}})$. Then $\bK := \PP(V^{*}\otimes \mathfrak{gl}_{2})\git \SL_{2} \times \SL_{2}$ is isomorphic to the connected component of $\rM_{\PP V}(v)$ containing $I_{\PP^{n-2}, Q}(1)$. 
\end{proposition}

\begin{proof}
Let $\rM_{\PP V}(v)^{c}$ (resp. $\rM_{\PP V}(v)^{m}$) be the connected (resp. irreducible) component of $\rM_{\PP V}(v)$ containing $I_{\PP^{n-2}, Q}(1)$. We will show that $\bK \cong \rM_{\PP V}(v)^{m} \cong \rM_{\PP V}(v)^{c}$. 

By Lemma \ref{lem:stabilitycomparison}, the universal family of quiver representations over $\PP (V^{*}\otimes \mathfrak{gl}_{2})^{ss}$ induces a morphism $f : \PP (V^{*}\otimes \mathfrak{gl}_{2})^{ss} \to \rM_{\PP V}(v)$ and $f$ is $\SL_{2} \times \SL_{2}$-invariant. Thus the map $f$ descends to the quotient map 
\[
	\bar{f} : \bK := \PP (V^{*}\otimes \mathfrak{gl}_{2})\git 
	\SL_{2}\times \SL_{2} \to \rM_{\PP V}(v).
\]
From the description of cokernels in Lemma \ref{lem:stabilitycomparison}, it is clear that $\bar{f}$ is injective. Furthermore, at a general stable point $[F] \in \im \;\bar{f}$, $\dim T_{[F]}\rM_{\PP V}(v) = \Ext^{1}(F, F) = 4n - 3 = \dim \bK$. Therefore $\im \;\bar{f} = \rM_{\PP V}(v)^{m}$. Since $\bK$ is normal and $\bar{f}$ is injective, $\bK$ is isomorphic to the normalization of $\rM_{\PP V}(v)^{m}$. 

Now it is sufficient to show that $\rM_{\PP V}(v)^{c}$ is irreducible and normal. From the standard construction of moduli spaces of semistable sheaves, $\rM_{\PP V}(v)$ is an $\SL_{P(m)}$-GIT quotient of the quot scheme $\mathrm{Quot}(\cH, P)$ where $\cH = W \otimes \cO_{\PP V}(-m)$, $\dim W = P(m) = P(F(m))$ for some $m \gg 0$. From the resolution \eqref{eqn:resolution}, it is straightforward to see that $\Ext^{2}(F, F) = 0$. Since $F$ is $m$-regular for some $m \gg 0$, we may assume that $\Ext^{1}(\cH, F) = 0$. Thus if we denote the kernel of $\cH \to F \to 0$ by $K$, then $\Ext^{1}(K, F) = 0$. This implies that the irreducible component of $\mathrm{Quot}(\mathcal{H}, P)^{ss}$ containing $\cH \to I_{\PP^{n-2}, Q}(1) \to 0$ is smooth and hence coincides with the connected component. In particular, its quotient, $\rM_{\PP V}(v)^{c}$, is irreducible and normal. Therefore $\rM_{\PP V}(v)^{c} = \rM_{\PP V}(v)^{m} = \bK$. 
\end{proof}

When $n = 3$, it was shown that $\rM_{\PP V}(v)$ is indeed irreducible (\cite[Proposition 3.6]{LP93b}). Since $P(F)(m) = m^{2}+3m+2$ for a semistable sheaf $F$ of class $v$, we will use the notation $\rM_{\PP V}(m^{2}+3m+2)$ for $\rM_{\PP V}(v)$ if it is better in the context. 

\begin{question}
Is there any extra component of $\rM_{\PP V}(v)$ if $n > 3$?
\end{question}

\subsection{Birational models of moduli spaces of rational curves in a Grassmannian}\label{ssec:KroneckerGrass}

When the dimension vector is $(2, d)$ where $d < n+1 = \dim V$, the moduli space of Kronecker $V$-modules $\PP\Hom(V \otimes \CC^{2}, \CC^{d})\git\SL_{2} \times \SL_{d}$ provides a birational model of $\overline{\rM}_{0,0}(\Gr(d, V^{*}), d)\cong \overline{\rM}_{0,0}(\Gr(n-d+1, V), d)$. This is a direct generalization of \cite[Section 6.1]{CM16}.

\begin{proposition}\label{prop:KroneckerKontsevichbirational}
There is a birational map 
\begin{equation}\label{eqn:KroneckerGrass}
	\Phi : \PP\Hom(V \otimes \CC^{2}, \CC^{d})\git \SL_{2} \times \SL_{d}
	\dashrightarrow \overline{\rM}_{0,0}(\Gr(d, V^{*}), d).
\end{equation}
\end{proposition}

\begin{proof}
Let $S \to \PP \Hom(V \otimes \CC^{2}, \CC^{d}) \cong \PP (\Hom(\CC^{2}, \CC^{d})\otimes V^{*})$ be a morphism. It induces a family of sheaf morphisms 
\[
	2\cO_{S \times \PP V}(-1) \stackrel{M}{\longrightarrow}
	d\cO_{S \times \PP V}.
\]
By taking the pull-back by the projection $q : S \times \PP V \times \PP^{1} \to S \times \PP V$, we obtain 
\[
	2\cO_{S \times \PP V \times \PP^{1}}(-1, 0)
	\stackrel{q^{*}M}{\longrightarrow}
	d\cO_{S \times \PP V \times \PP^{1}}.
\]
By composing with the tautological inclusion $\iota : \cO_{S \times \PP V \times \PP^{1}}(-1, -1) \to 2\cO_{S \times \PP V \times \PP^{1}}(-1, 0)$, we have
\[
	\cO_{S \times \PP V \times \PP^{1}}(-1, -1) 
	\stackrel{q^{*}M \circ \iota}{\longrightarrow}
	d\cO_{S \times \PP V \times \PP^{1}}.
\]
Take the dual
\[
	d\cO_{S \times \PP V \times \PP^{1}}
	\stackrel{(q^{*}M \circ \iota)^{*}}{\to}
	\cO_{S \times \PP V \times \PP^{1}}(1, 1),
\]
take the push-forward $p_{*}$ for $p : S \times \PP V \times \PP^{1} \to S \times \PP^{1}$, and finally take the tensor product with $\cO_{S \times \PP^{1}}(-1)$, we have:
\begin{equation}\label{eqn:familyofmaps}
	d\cO_{S \times \PP^{1}}(-1)
	\stackrel{p_{*}((q^{*}M \circ \iota)^{*}) \otimes \cO_{\PP^{1}}(-1)}
	{\longrightarrow}
	V^{*} \otimes \cO_{S \times \PP^{1}}.
\end{equation}
For a general point, it defines a family of rank $d$, degree $d$ subbundle of a trivial bundle $V^{*} \otimes \cO_{\PP^{1}}$. Thus we obtain a family of stable maps to $\Gr(d, V^{*})$ (or equivalently, to $\Gr(n-d+1, V)$). 

So we have a rational map $\PP\Hom (V \otimes \CC^{2}, \CC^{d}) \dashrightarrow \overline{\rM}_{0,0}(\Gr(d, V^{*}), d)$. This map is (on the domain) $\SL_{2} \times \SL_{d}$-equivariant since $\SL_{2}$ acts as the change of coordinates of $\PP^{1}$, and $\SL_{d}$ acts as the change of coordinates of $d\cO_{\PP^{1}}(-1)$. Therefore the rational map induces a quotient map
\[
	\Phi : \PP \Hom(V \otimes \CC^{2}, \CC^{d}) \git \SL_{2} \times \SL_{d}
	\dashrightarrow
	\overline{\rM}_{0,0}(\Gr(d, V^{*}), d).
\]
This map is birational since a general balanced, non-degenerate stable map can be obtained from a unique stable Kronecker module. 
\end{proof}

In the next section, we will show that when $d = 2$, the inverse of $\Phi$ is indeed a partial desingularization in the sense of Kirwan (\cite{Kir85}). In particular,  $\Phi^{-1}$ is regular. In general, $\Phi^{-1}$ is a rational contraction. 

\begin{question}
For $d \ge 3$, is $\Phi^{-1}$ regular?
\end{question}

\subsection{Moduli of Kronecker modules as a double cover}\label{ssec:Kroneckerdblcover}

Fix a natural number $n \ge 3$. In this section, we show that the contraction $T_{4}$ of $\bH$ in Section \ref{ssec:determinantalvarieties} is isomorphic to a moduli space of Kronecker modules. 

\begin{proposition}
Let $V$ be an ($n+1$)-dimensional vector space. Then
\[
	\bK := \PP (V^{*}\otimes \mathfrak{gl}_{2})\git \SL_{2} \times \SL_{2} 
	\cong T_{4}.
\]
\end{proposition}

\begin{proof}
There is the determinant map 
\[
	\det : \bK \to \PP(\Sym^{2}V^{*}), 
\]
which maps $M$ to $\det M$. The image of $\det$ is exactly $S_{4}$. It is straightforward to check that
\begin{enumerate}
\item $\det$ is finite, and generically two-to-one since $\det M = \det M^{t}$ and $M \not\equiv M^{t}$;
\item it is ramified along $S_{3}$. 
\end{enumerate}
Let 
\[
	U_{4} := \{([\Pi], [Q])\;|\; \PP(\Pi) \subset Q\} \subset 
	\Gr(n-1, V) \times \PP(\Sym^{2}V^{*})
\]
be the incidence variety. There is a morphism $\pi : U_{4} \to T_{4}$ with connected fibers. Consider the $(\CC^{2}-\{0\})^{2}\setminus \Delta$-bundle
\[
	E_{4} := \{(\ell_{1}, \ell_{2}, [\Pi], [Q])\;|\; \ell_{i}|_{\Pi} = 0,
	\ell_{i} \ne 0, \PP(\Pi) \subset Q\} \subset (V^{*})^{2} \times 
	\Gr(n-1, V) \times \PP(\Sym^{2}V^{*})
\]
over $U_{4}$. Let $f \in \Sym^{2}V^{*}$ be the defining equation of $Q$. Then $f = m_{1}\ell_{1} + m_{2}\ell_{2}$ for $m_{i} \in V^{*}$ and $m_{i}$'s are defined uniquely up to scalar multiple. There is a morphism $m : E_{4}
 \to \bK$ where 
\[
	m(\ell_{1}, \ell_{2}, [\Pi], [Q]) = \left[\begin{array}{cc}
	\ell_{1} & -m_{2}\\ \ell_{2} & m_{1}\end{array}\right].
\]
Note that this map is well-defined since a scalar multiple to the second column defines the same point in the quotient. 

Now it is straightforward to check that $m$ descends to $U_{4}$ and to $T_{4}$, so we obtain a map $\overline{m} : T_{4} \to \bK$, which is an $S_{4}$-morphism. Since both $T_{4}$ and $\bK$ are two-to-one to $S_{4}$, $\overline{m}$ is bijective. It is an isomorphism since it is a bijective morphism between two normal varieties.
\end{proof}


\section{Partial desingularization}\label{sec:pardesing}

When $d = 2$, the birational map $\Phi^{-1}$ in \eqref{eqn:KroneckerGrass} is indeed a regular contraction. Furthermore, it can be understood as the partial desingularization in the sense of \cite{Kir85}. In this section, we prove the following result. Let $V$ be a fixed $(n+1)$-dimensional vector space. 

\begin{theorem}\label{thm:Kirwandes}
The partial desingularization of $\bK := \PP(V^{*}\otimes \mathfrak{gl}_{2})\git \SL_{2} \times \SL_{2}$ is isomorphic to $\bM := \overline{\rM}_{0,0}(\Gr(n-1, V), 2)$. 
\end{theorem}

This result is a generalization of \cite[Section 6]{CM16} with minor modifications. For the reader's convenience, we give a detailed proof here.

\subsection{GIT stratification on $\bX^{ss}$}

Let $\bX := \PP(V^{*}\otimes \mathfrak{gl}_{2})$. By using the description of semistable locus in Lemma \ref{lem:stability}, we can define a stratification 
\[
	\bX^{ss} = \bY_{0} \sqcup \bZ_{0} \sqcup \bY_{1} \sqcup 
	\bZ_{1} \sqcup \bX^{s},
\]
as the following. For notational simplicity, let $G = \SL_{2} \times \SL_{2}$. 

Let $\bY_{0} \subset \bX^{ss}$ be the locus of matrices equivalent to scalar matrices. More precisely, let $\bY_{0}'$ be the image of 
\[
	\rho_{0} : \PP(\mathfrak{gl}_{2}) \times \PP V^{*} \to \bX, 
	\quad (A, g) \mapsto A\left[\begin{array}{cc}g&0\\0&g\end{array}\right]
\]
and let $\bY_{0} := \bY_{0}' \cap \bX^{ss}$. Then $\bY_{0} = \rho_{0}(\mathrm{PGL}_{2} \times \PP V^{*})$ and $\rho_{0}$ is an embedding on this locus. Thus $\bY_{0}$ is an ($n+3$)-dimensional smooth closed subvariety. At each closed point $M = g\cdot \mathrm{Id} \in \bY_{0}$, the normal bundle $N_{\bY_{0}/\bX^{ss}}|_{M}$ is naturally isomorphic to $H \otimes \mathfrak{sl}_{2}$, where $H \cong V^{*}/\langle g\rangle$ is an $n$-dimensional quotient of $V^{*}$. 

Let $\bZ_{0} \subset \bX^{ss}$ be the locus of matrices equivalent to upper triangular matrices whose diagonal entries are proportional to each other. Formally, we can define $\bZ_{0}$ as the following. Let $\bZ_{0}'$ be the image of 
\[
	\tau_{0} : G \times \PP ((V^{*})^{2}) \to \bX, \quad
	((A, B), (g, k)) \mapsto 
	A\left[\begin{array}{cc}g & k\\0 &g\end{array}\right]B^{-1}.
\]
Then $\bZ_{0} = (\bZ_{0}' \cap \bX^{ss})\setminus \bY_{0}$. Let $\overline{\bZ}_{0} := \bZ_{0} \sqcup \bY_{0}$, the closure of $\bZ_{0}$ in $\bX^{ss}$. A general fiber of $\tau_{0}$ is 3-dimensional, so $\bZ_{0}$ is a $(2n+4)$-dimensional subvariety. The normal cone $C_{\bY_{0}/\overline{\bZ}_{0}}$ is an analytic locally trivial bundle, whose fiber at $M = g\cdot \mathrm{Id} \in \bY_{0}$ is isomorphic to $\mathrm{Stab} \;M\cdot (H \otimes \langle e\rangle) = \mathrm{Stab} \;M \cdot (H \otimes \langle f\rangle)$. Here $\{h, e, f\}$ is the standard basis of $\mathfrak{sl}_{2}$. In $\PP(N_{\bY_{0}/\overline{\bZ}_{0}}|_{M}) \cong \PP(H \otimes \mathfrak{sl}_{2})$, $\PP(C_{\bY_{0}/\overline{\bZ}_{0}}|_{M}) \cong \PP H \times \mathrm{PGL}_{2}\cdot \PP \langle e \rangle \cong \PP H \times \PP^{1}$.

Let $\bY_{1} \subset \bX^{ss}$ be the locus of matrices equivalent to non-scalar diagonal matrices. We may impose the scheme structure to $\bY_{1}$ as the following way. Let $\bY_{1}'$ be the image of 
\[
	\rho_{1} : G \times \PP((V^{*})^{2}) \to \bX, \quad
	((A, B), (g, k)) \mapsto 
	A \left[\begin{array}{cc}g&0\\0&k\end{array}\right] B^{-1}.
\]
Let $\bY_{1} = (\bY_{1}' \cap \bX^{ss})\setminus \bY_{0}$, and let $\overline{\bY}_{1} := \bY_{0} \sqcup \bY_{1}$ be the closure of $\bY_{1}$ in $\bX^{ss}$. Then $\bY_{1}$ is an irreducible $G$-invariant smooth variety of dimension $2n + 5$, and $\overline{\bY}_{1}$ is singular along $\bY_{0}$. At a closed point $M = \left[\begin{array}{cc}g&0\\0&k\end{array}\right] \in \bY_{1}$, the normal bundle $N_{\overline{\bY}_{1}/\bX^{ss}}|_{M}$ is isomorphic to $K \otimes \langle e, f\rangle$ where $K = V^{*}/\langle g, k\rangle$. For $M = g \cdot \mathrm{Id} \in \bY_{0}$, the normal cone $C_{\bY_{0}/\overline{\bY}_{1}}|_{M}$ is isomorphic to $\mathrm{Stab}\;M \cdot (H \otimes \langle h \rangle) \subset H \otimes \mathfrak{sl}_{2} \cong N_{\bY_{0}/\bX^{ss}}|_{M}$ and its projectivization in $\PP(H \otimes \mathfrak{sl}_{2})$ is isomorphic to $\PP H \times \overline{\mathrm{PGL}_{2}\cdot \PP\langle h \rangle} \cong \PP H \times \PP^{2}$. Note that although $\bY_{1} \cap \bZ_{0} = \emptyset$, $\PP (C_{\bY_{0}/\overline{\bZ}_{0}}|_{M}) \subset \PP(C_{\bY_{0}/\overline{\bY}_{1}}|_{M})$ because the normal cone to $\overline{\bZ}_{0}$ is tangent to the normal cone to $\overline{\bY}_{1}$. 

Finally, let $\bZ_{1} \subset \bX^{ss}$ be the locus of matrices equivalent to upper triangular matrices. If we denote the image of 
\[
	\tau_{1} : G \times \PP((V^{*})^{3}) \to \bX, \quad
	((A, B), (g, k, \ell)) \mapsto 
	A \left[\begin{array}{cc}g&k \\ 0 &\ell\end{array}\right] B^{-1}
\]
by $\bZ_{1}'$, $\bZ_{1} = (\bZ_{1}' \cap \bX^{ss}) \setminus (\overline{\bY}_{1} \sqcup \bZ_{0})$. Then $\overline{\bZ}_{1} = \bZ_{1} \sqcup \bY_{1} \sqcup \bZ_{0} \sqcup \bY_{0}$. We may check that $\bZ_{1}$ is an irreducible $(3n+4)$-dimensional $G$-invariant variety. The normal cone $C_{\overline{\bY}_{1}/\overline{\bZ}_{1}}|_{\bY_{1}}$ is an analytic fiber bundle whose fiber is the union of two disjoint rank $n-1$ subbundles of $N_{\overline{\bY}_{1}/\bX^{ss}}|_{\bY_{1}}$. At a closed point $M = \left[\begin{array}{cc}g&0\\0&k\end{array}\right] \in \bY_{1}$, $C_{\overline{\bY}_{1}/\overline{\bZ}_{1}}|_{M} \cong K \otimes \langle e\rangle \sqcup K \otimes \langle f\rangle \subset K \otimes \langle e, f\rangle \cong N_{\overline{\bY}_{1}/\bX^{ss}}|_{M}$. 

\subsection{Kirwan's partial desingularization}\label{ssec:partialdesing}

In this section, we describe Kirwan's partial desingularization of $\bK := \bX\git G$. For the general construction and its proof, consult \cite{Kir85}. Let $\bX^{0} := \bX^{ss}$. In $\bX^{0}$, the deepest stratum with the largest stabilizer group is $\bY_{0}$. Let $\pi_{1}' :{ \bX^{1}}' \to \bX^{0}$ be the blow-up of $\bX^{0}$ along $\bY_{0}$. Let $\bY_{0}^{1}$ be the exceptional divisor, and let $\overline{\bY}_{1}^{1}$, $\overline{\bZ}_{i}^{1}$ be the proper transform of $\overline{\bY}_{1}$, $\overline{\bZ}_{i}$, respectively. Since the normal cone $C_{\bY_{0}/\overline{\bY}_{1}}$ is a cone over a smooth variety and $\bY_{1}$ is smooth, $\overline{\bY}_{1}^{1}$ is a smooth subvariety of ${\bX^{1}}'$. 

Since $\rho(\bX) = 1$, there is a unique linearization $L_{0}$ on $\bX$, up to scaling. Let $L_{1} := {\pi_{1}'}^{*}(L_{0}) \otimes \cO(-\epsilon_{1} \bY_{0}^{1})$ for some $0 < \epsilon_{1} \ll 1$. Then $L_{1}$ is an ample line bundle with a linearized $G$-action. With respect to this linearization, $\overline{\bZ}_{0}^{1}$ is unstable since any orbit in $\bZ_{0}$ is not closed in $\bX^{0}$ (\cite[Lemma 6.6]{Kir85}). Let $\bX^{1} := {\bX^{1}}' \setminus \overline{\bZ}_{0}^{1}$ and let $\pi_{1} : \bX^{1} \to \bX^{0}$ be the restriction of $\pi_{1}'$. 

Similarly, let $\pi_{2}' : {\bX^{2}}' \to \bX^{1}$ be the blow-up of $\bX^{1}$ along $\overline{\bY}_{1}^{1} \cap \bX^{1}$. Since $\overline{\bY}_{1}^{1} \cap \bX^{1}$ is smooth, ${\bX^{2}}'$ is also smooth. Let $\overline{\bY}_{1}^{2}$ be the exceptional divisor. And let $\bY_{0}^{2}$, $\overline{\bZ}_{1}^{2}$ be the proper transform of $\bY_{0}^{2} \cap \bX^{1}$, $\overline{\bZ}_{1}^{1}\cap \bX^{1}$, respectively. Let $L_{2} := \pi_{2}^{*}(L_{1}) \otimes \cO(-\epsilon_{2}\overline{\bY}_{1}^{2})$ for some $0 < \epsilon_{2} \ll \epsilon_{1}$. Then $L_{2}$ is ample. Furthermore, since $\overline{\bY}_{1}^{1} \cap \bX^{1}$ is a $G$-invariant subvariety, $L_{2}$ inherits a linearized $G$-action, too. With respect to this $G$-action, $\overline{\bZ}_{1}^{2}$ is unstable. Let $\bX^{2} := {\bX^{2}}'\setminus \overline{\bZ}_{1}^{2}$ and let $\pi_{2} : \bX^{2} \to \bX^{1}$ be the restriction of $\pi_{2}'$. 

Note that on $\bX^{2}$, every point has a finite stabilizer. Therefore $\bX^{2} = (\bX^{2})^{ss} = (\bX^{2})^{s}$. The partial desingularization of $\bX \git G$ is $\bX^{2}\git_{L}G = \bX^{2}/G$. The blow-up morphisms $\pi_{1}$, $\pi_{2}$ induce quotient maps $\overline{\pi}_{1}$ and $\overline{\pi}_{2}$. In summary, we obtain the following commutative diagram:
\[
	\xymatrix{\bX^{2} \ar[d]^{/G} \ar[r]^{\pi_{2}}_{\overline{\bY}_{1}^{1}\cap \bX^{1}}
	& \bX^{1} \ar[d]^{/G} \ar[r]^{\pi_{1}}_{\bY_{0}}
	& \bX^{0}\ar[d]^{/G}\\
	\bX^{2}/G \ar[r]^{\overline{\pi}_{2}} 
	& \bX^{1}\git G \ar[r]^{\overline{\pi}_{1}}
	& \bX\git G}
\]
Let $\pi := \pi_{1} \circ \pi_{2}$, and let $\overline{\pi} := \overline{\pi}_{1} \circ \overline{\pi}_{2}$ be its quotient map. Note that the partial desingularization $\bX^{2}/G$ has only finite quotient singularities only, since every point on $\bX^{2}$ has the finite stabilizer. 

During the desingularization process, we can keep track the change of the GIT quotient. For $M \in \bY_{0}$, ${\pi_{1}'}^{-1}(M) \cong \PP(H \otimes \mathfrak{sl}_{2})$ where $H$ is an $n$-dimensional quotient of $V^{*}$, because it is the projectivized normal cone. On the fiber ${\pi_{1}'}^{-1}(M)$, there is an induced $\mathrm{Stab}\;M \cong \SL_{2} \ltimes \ZZ_{2}$-action, which is induced by a trivial action on $H$ and the standard $\SL_{2}$-action on $\mathfrak{sl}_{2}$. Also the $\ZZ_{2}$ acts trivially. The unstable locus is precisely $\PP(C_{\bY_{0}/\overline{\bZ}_{0}}|_{M}) \cong \PP H \times \PP^{1}$. Therefore in $\bX^{1} = \bX^{1}\setminus \overline{\bZ}_{0}^{1}$, the inverse image of $M$ is $\PP(H \otimes \mathfrak{sl}_{2})^{ss}$. If we denote the image of $M$ in $\bX\git G$ by $\overline{M}$, then 
\[
	\overline{\pi}_{1}^{-1}(\overline{M}) \cong 
	\PP(H\otimes \mathfrak{sl}_{2})^{ss}/\mathrm{Stab}\;M
	\cong \PP(H\otimes \mathfrak{sl}_{2})\git \SL_{2}.
\]
The locus of strictly semistable points with closed orbits on $\PP(H \otimes \mathfrak{sl}_{2})^{ss}$ is isomorphic to $(\PP H \times \PP \mathfrak{sl}_{2})^{ss}$, which is precisely $\PP(C_{\bY_{0}/\overline{\bY}_{1}})^{ss}$. (Indeed the strictly semistable locus is $\PP(C_{\bY_{0}/\overline{\bZ}_{1}})^{ss}$.) Thus on the fiber of $\overline{M}$, the second blow-up $\overline{\pi}_{2} : \bX^{2}\git G \to \bX^{1}\git G$ is the partial desingularization of the fiber $\PP(H \otimes \mathfrak{sl}_{2})\git \SL_{2}$. In \cite[Theorem 4.1]{Kie07}, it was shown that the partial resolution is isomorphic to the moduli space $\overline{\rM}_{0,0}(\PP H, 2)$ of degree two stable maps to $\PP H$.

For $M \in \bY_{1}^{1}$, the inverse image ${\pi_{2}'}^{-1}(M)$ is isomorphic to $\PP(K \otimes \langle e,f\rangle)$. On this normal bundle, there is an induced $\mathrm{Stab}\;M \cong \CC^{*}\ltimes \ZZ_{2}$-action. The unstable locus is precisely $\PP(K \otimes \langle e\rangle) \sqcup \PP(K \otimes \langle f\rangle)$ and there is no strictly semi-stable point. Thus on $\bX^{2}$, $\pi_{2}^{-1}(M) \cong \PP(K \otimes \langle e, f\rangle)^{s}$. Therefore in $\bX^{2}/G$, 
\[
	\overline{\pi}_{2}^{-1}(\overline{M}) \cong 
	\PP(K \otimes \langle e, f\rangle)\git \CC^{*}\ltimes \ZZ_{2}
	\cong \PP^{n-2} \times \PP^{n-2}\git \ZZ_{2}
	\cong (\PP^{n-2})^{2}.
\]
Note that $\ZZ_{2}$ acts trivially on the projectivized normal cone. 

\subsection{Elementary modification of maps}

In this section, we prove Theorem \ref{thm:Kirwandes}.

\begin{proof}[Proof of Theorem \ref{thm:Kirwandes}]
Let $\bX^{0} := \PP (V^{*}\otimes \mathfrak{gl}_{2})^{ss}$. By taking the dual of \eqref{eqn:familyofmaps}, we have 
\[
	V\otimes \cO_{\bX^{0} \times \PP^{1}}
	\stackrel{\wedge^{2}(p_{*}((q^{*}M \circ \iota)^{*}) \otimes 
	\cO_{\PP^{1}}(-1))^{*}}{\longrightarrow} 
	2\cO_{\bX^{0} \times \PP^{1}}(1).
\]
It induces a bundle morphism 
\[
	\wedge^{2}V\otimes \cO_{\bX^{0} \times \PP^{1}}
	\stackrel{\wedge^{2}(p_{*}((q^{*}M \circ \iota)^{*}) \otimes 
	\cO_{\PP^{1}}(-1))^{*}}{\longrightarrow} 
	\cO_{\bX^{0} \times \PP^{1}}(2).
\]
Since this map is surjective on $\bX^{s} \times \PP^{1}$, we obtain a rational map 
\[
	f_{0} : \bX^{0} \times \PP^{1} \dashrightarrow \Gr(n-1, V)
	\hookrightarrow \PP(\wedge^{n-1}V) \cong \PP(\wedge^{2}V^{*}).
\]
which is regular on $\bX^{s} \times \PP^{1}$.

Let $F_{0} := \wedge^{2}(p_{*}((q^{*}M \circ \iota)^{*}) \otimes \cO_{\PP^{1}}(-1))^{*}$. Let $\pi_{1} \times \mathrm{id} : \bX^{1} \times \PP^{1} \to \bX^{0} \times \PP^{1}$ be the blow-up morphism. All sections giving $F_{0}$ simultaneously vanish along $\bY_{0}^{1} \times \PP^{1}$. So $f_{0}$ is not defined on $\bY_{0}^{1} \times \PP^{1}$. But this implies that the pull-back morphism $(\pi_{1} \times \mathrm{id})^{*}F_{0}$ factor through 
\[
	\wedge^{2}V\otimes \cO_{\bX^{0} \times \PP^{1}} 
	\stackrel{F_{1}:= (\pi_{1} \times \mathrm{id})^{*}F_{0}}
	{\longrightarrow}
	\cO_{\bX^{1} \times \PP^{1}}(2)(-\bY_{0}^{1}).
\]
Therefore we obtain an extended family $f_{1}$ of rational maps over $\bX^{1}$
\[
	\xymatrix{\bX^{1} \times \PP^{1} \ar[d]^{\pi_{1} \times \PP^{1}}
	\ar@{-->}[rrd]^{f_{1}}\\
	\bX^{0} \times \PP^{1} \ar@{-->}[r]^{f_{0}} &
	\Gr(n-1, V) \ar[r] & \PP(\wedge^{2}V^{*})}
\]
whose undefined locus is precisely a two-to-one \'etale cover of $\overline{\bY}_{1}^{1}$ because for each point $M \in \overline{\bY}_{1}^{1}$, the undefined locus of $f_{1}$ restricted to $\{M\} \times \PP^{1}$ is two distinct points. 

Let $\pi_{2} \times \mathrm{id} : \bX^{2} \times \PP^{1} \to \bX^{1} \times \PP^{1}$ be the second blow-up morphism. The base locus $\mathbf{B}$ of $(\pi_{2} \times \mathrm{id})^{*}f_{1}$ is a two-to-one \'etale cover of $\overline{\bY}_{1}^{2}$, so it is a smooth codimension two subvariety of $\bX^{2} \times \PP^{1}$. Let $\sigma : \Gamma \to \bX^{2} \times \PP^{1}$ be the blow-up along $\mathbf{B}$. Let $\mathbf{E}$ be the exceptional divisor. The composition $s : \Gamma \to \bX^{2} \times \PP^{1} \to \bX^{2}$ is a flat family of rational curves. Moreover, the pull-back morphism $\sigma^{*}(\pi_{2}\times \mathrm{id})^{*}F_{1}$
\[
	\wedge^{2}V \otimes \cO_{\Gamma}
	\stackrel{\sigma^{*}(\pi_{2}\times \mathrm{id})^{*}F_{1}}
	{\longrightarrow}
	\sigma^{*}\cO_{\bX^{2} \times \PP^{1}}(2)(-\bY_{0}^{2})
\]
factors through
\[
	\wedge^{2}V \otimes \cO_{\Gamma}
	\stackrel{F_{2}}{\longrightarrow}
	\sigma^{*}\cO_{\bX^{2} \times \PP^{1}}(2)(-\bY_{0}^{2}-\mathbf{E}).
\]
Now $F_{2}$ is surjective and we obtain a regular morphism
$f_{2} : \Gamma \to \PP(\wedge^{2}V^{*})$. 
\[
	\xymatrix{\Gamma \ar[d]^{\sigma} \ar[rrddd]^{f_{2}}\\
	\bX^{2} \times \PP^{1} \ar[d]^{\pi_{2} \times \mathrm{id}}\\
	\bX^{1} \times \PP^{1} \ar[d]^{\pi_{1} \times \mathrm{id}}
	\ar@{-->}[rrd]^{f_{1}}\\
	\bX^{0} \times \PP^{1} \ar@{-->}[r]^{f_{0}} &
	\Gr(n-1, V) \ar[r] & \PP(\wedge^{2}V^{*})}
\]
So we have a flat family of maps $(s : \Gamma \to \bX^{2}, f_{2} : \Gamma \to \PP(\wedge^{2}V^{*}))$ over $\bX^{2}$. By stabilizing, we obtain a family of stable maps $(\overline{s} : \overline{\Gamma} \to \bX^{2}, \overline{f}_{2} : \overline{\Gamma} \to \PP(\wedge^{2}V^{*}))$. Clearly $\overline{f}_{2}$ factors through $\Gr(n-1, V)$ since it does on an open dense subset. Therefore we have a map $\Psi : \bX^{2} \to \bM$, which is $G$-invariant from the construction. Thus we obtain the quotient map $\overline{\Psi} : \bX^{2}/G \to \bM$. This is an isomorphism since it is a birational morphism between two $\QQ$-factorial normal varieties with the same Picard number, which is 3. 
\end{proof}

Here we leave two explicit examples of the elementary modification of a family of maps over a curve. 

\begin{example}
Let $S$ be a small disk in $\CC$ containing $0$ (or the $\spec$ of a discrete valuation ring) and let 
\begin{eqnarray*}
	g : S &\to& \PP(V^{*}\otimes \mathfrak{gl}_{2})\\
	\lambda & \mapsto & \left[\begin{array}{cc}
	x_{0} & \lambda (\sum_{i=1}^{n}a_{i}x_{i})\\
	\lambda(\sum_{i=1}^{n}b_{i}x_{i}) & x_{0}\end{array}\right].
\end{eqnarray*}
Then the associated map 
\[
	V \otimes \cO_{S \times \PP^{1}} \stackrel{G}{\to} 
	2\cO_{S \times \PP^{1}}(1)
\]
is given by a matrix
\[
	G = \left[\begin{array}{ccccc}
	s & \lambda a_{1}t &\lambda a_{2}t &\cdots &\lambda a_{n}t\\
	t & \lambda b_{1}s & \lambda b_{2}s & \cdots &\lambda b_{n}s
	\end{array}\right]
\]
where $[s:t]$ is the homogeneous coordinate of $\PP^{1}$. Note that this family of maps is not surjective when $\lambda = 0$. By taking the wedge product, we obtain a family of maps $\wedge^{2}V \otimes \cO_{S \times \PP^{1}} \stackrel{F_{0}:=\wedge^{2}G}{\longrightarrow} \cO_{S \times \PP^{1}}(2)$ where 
\[
	[F_{0}]_{I} = \begin{cases}
	\lambda (b_{i}s^{2} - a_{i}t^{2}), & I = \{0, i\}\\
	\lambda^{2} (a_{i}b_{j} - a_{j}b_{i})st, & I = \{i, j\}, 0 \notin I
	\end{cases}.
\]
Thus if we take the map 
\[
	\lambda^{2}V \otimes \cO_{S \times \PP^{1}}
	\stackrel{F_{1}}{\to} \cO_{S \times \PP^{1}}(2)(-\bY_{0}^{1}), 
\]
it is given by 
\[
	[F_{1}]_{I} = \begin{cases}
	(b_{i}s^{2} - a_{i}t^{2}), & I = \{0, i\}\\
	\lambda (a_{i}b_{j} - a_{j}b_{i})st, & I = \{i, j\}, 0 \notin I
	\end{cases}.
\]
When $\lambda = 0$, we could recover the modified map $G'(0)$ so that $\wedge^{2}G'(0) = F_{1}(0)$, which is 
\[
	V \otimes \cO_{S(0) \times \PP^{1}} 
	\stackrel{G'(0)}{\longrightarrow} \cO_{S(0) \times \PP^{1}}
	\oplus \cO_{S(0) \times \PP^{1}}(2),
\]
and 
\[	
	G'(0) = \left[\begin{array}{ccccc}
	1 & 0 & 0& \cdots & 0\\
	0 & b_{1}s^{2} - a_{1}t^{2} & b_{2}s^{2} - a_{2}t^{2} &\cdots
	& b_{n}s^{2} - a_{n}t^{2}
	\end{array}\right].
\]
Since the image has $\cO_{S(0) \times \PP^{1}}$ factor, the subbundle $\ker G'(0)$ is degenerated. Therefore the modified map is in $D_{\deg}$. 
\end{example}

\begin{example}
Let 
\begin{eqnarray*}
	h : S &\to& \PP(V^{*}\otimes \mathfrak{gl}_{2})\\
	\lambda & \mapsto & \left[\begin{array}{cc}
	x_{0} & \lambda(\sum_{i=2}^{n}a_{i}x_{i})\\
	\lambda(\sum_{i=2}^{n}b_{i}x_{i}) & x_{1}\end{array}\right]
\end{eqnarray*}
be a family over a small disk $S$. The associated map $V \otimes \cO_{S \times \PP^{1}}\stackrel{H}{\to} 2\cO_{S \times \PP^{1}}(1)$ is given by 
\[
	H = \left[\begin{array}{ccccc}
	s &0 & \lambda a_{2}t & \cdots & \lambda a_{n}t \\
	0 & t & \lambda b_{2}s & \cdots & \lambda b_{n}s
	\end{array}\right].
\]
For $F_{0}:= \wedge^{2}H$, 
\[
	[F_{0}]_{I} = \begin{cases}
	st, & I = \{0, 1\}\\
	\lambda b_{i}s^{2}, & I = \{0, i\}, i \ge 2\\
	-\lambda a_{i}t^{2}, & I = \{1, i\}, i \ge 2\\
	\lambda^{2}(a_{i}b_{j}- a_{j}b_{i}), & I = \{i, j\}, i, j \ge 2.
	\end{cases}
\]
When $\lambda = 0$, $H$ is not surjective at two points $[0:1]$ and $[1:0]$, and except those two points, the map is constant. 

At $\lambda = s = 0$, take the blow-up and let $E_{1}\cong \PP^{1}$ be the exceptional divisor. On $E_{1}$ (with homogeneous coordinate $[s:\lambda]$), we obtain an extended degree one map $F_{2}^{1}(0)$ given by 
\[
	[F_{2}^{1}(0)]_{I} = \begin{cases}
	s, & I = \{0, 1\}\\
	-\lambda a_{i}, & I = \{1, i\}, i \ge 2\\
	0, & \mbox{ otherwise}
	\end{cases}.
\]
Thus the map $H^{1}(0) : V\otimes \cO_{\PP^{1}}\to \cO_{\PP^{1}}(1) \oplus \cO_{\PP^{1}}$ so that $\wedge^{2}H^{1}(0) = F_{2}^{1}(0)$, is given by 
\[
	H^{1}(0) = \left[\begin{array}{ccccc}
	s & 0 & a_{2}\lambda &\cdots &a_{n}\lambda \\
	0 & 1 &0 &\cdots &0
	\end{array}\right].
\]
Similarly, at $\lambda = t = 0$, we can compute the map $H^{2}(0)$ on the second exceptional divisor $E_{2}$:
\[
	H^{2}(0) = \left[\begin{array}{ccccc}
	1 & 0 & 0& \cdots &0\\
	0 & t & b_{2} & \cdots & b_{n}
	\end{array}\right].
\]
The stabilization at $\lambda = 0$ contracts central constant component. Therefore we obtain a limit $\lambda \to 0$ which is a stable map from a nodal curve.
\end{example}

The proof of the theorem and the above two examples shows the following corollary. 

\begin{corollary}\label{cor:partial}
\begin{enumerate}
\item The partial desingularization map $d : \bM \to \bK$ contracts $\Delta$ and $D_{\deg}$. For a point $[2H]$ of $d(D_{\deg}) \cong \PP V^{*}$, its fiber is isomorphic to the moduli space of stable maps $\overline{\rM}_{0,0}(\PP H, 2)$.
\item The second step of the partial desingularization $c : \bM \to \bX^{1}\git G$ contracts $\Delta$. This morphism maps a singular stable map to the union of two hyperplanes in $\PP V$ where each of them is the envelope of the irreducible component of the stable map. For a general point $[H \cup H']$ of $c(\Delta)$, the fiber is $(\PP (H \cap H')^{*})^{2}$.
\end{enumerate}
\end{corollary}

\begin{remark}\label{rem:intermtoquasimap}
Note that in item (2), the image $c(f)$ of a stable map $f : C_{1} \cup C_{2} \to \Gr(n-1, V)$ remembers not only $f(C_{1} \cap C_{2})$, but also $\langle f(C_{i})\rangle \subset \PP V$. On the other hand, the morphism $\bM \to \bU$ in Definition \ref{def:quasimap} maps a singular stable map $f : C_{1} \cup C_{2} \to \Gr(n-1, V)$ to a point $C_{1} \cap C_{2} \in \Gr(n-1, V)$. By rigidity lemma, we obtain a morphism $\bX^{1}\git G \to \bU$. 
\end{remark}


\section{Mori's program}\label{sec:Moriprogram}

In this section, we prove our main theorem (Theorem \ref{thm:mainintro}) and complete Mori's program for $\bM := \MzzGrt$. 

We start with a simple but useful observation. 

\begin{lemma}\label{lem:duality}
Let $\Phi : \bM := \MzzGrt \to \overline{\rM}_{0,0}(\Gr(2, V^{*}), 2)$ be an isomorphism induced by $\phi : \Gr(n-1, V) \cong \Gr(2, V^{*})$. Then $\Phi_{*}$ induces a reflection along the vertical line connecting $\Delta$ and $P$ in Figure \ref{fig:stablebaselocusdecomposition}. In other words, $\Phi_{*}(D_{\deg}) = D_{\mathrm{unb}}$, $\Phi_{*}(D_{\mathrm{unb}}) = D_{\deg}$, $\Phi_{*}(H_{\sigma_{1, 1}}) = H_{\sigma_{2}}$, $\Phi_{*}(H_{\sigma_{2}}) = H_{\sigma_{1,1}}$, and $\Phi_{*}(\Delta) = \Delta$. 
\end{lemma}

\begin{proof}
It follows from the induced isomorphism $\phi_{*} : \rA^{*}(\Gr(n-1, V)) \to \rA^{*}(\Gr(2, V^{*}))$ such that $\phi_{*}(\sigma_{1,1}) = \sigma_{2}$, $\phi_{*}(\sigma_{2}) = \sigma_{1,1}$, $\phi_{*}(\sigma_{(1,1)^{*}}) = \sigma_{(2)^{*}}$, and $\phi_{*}(\sigma_{(2)^{*}}) = \sigma_{(1,1)^{*}}$. $\Phi_{*}(\Delta) = \Delta$ is clear.
\end{proof}

\subsection{$n = 3$ case}

The first non-trivial case is $n = 3$, where $\Gr(n-1, V) = \Gr(2, 4)$. In this case, because of the self-duality of $\Gr(2, 4)$, the complete description is particularly clear. Essentially all of birational models in this case have been described in \cite{CC10, CM16}. For the reader's convenience, we leave the statement and references. 

\begin{theorem}\label{thm:n=3case}
Let $V$ be a vector space of dimension 4. For an effective divisor $D$ on $\bM := \overline{\rM}_{0,0}(\Gr(2, V), 2)$, 
\begin{enumerate}
\item If $D = a H_{\sigma_{1,1}} + bH_{\sigma_{2}} + cT$ for $a, b, c > 0$, then $\bM(D) \cong \bM$.
\item If $D = aH_{\sigma_{1, 1}} + bH_{\sigma_{2}} + cP$ for $a, b, c > 0$, then $\bM(D) \cong \bH := \Hilb^{2m+1}(\Gr(2, V))$. 
\item If $D = aH_{\sigma_{2}} + bD_{\deg} + c\Delta$ for $a > 0$ and $b, c \ge 0$, then $\bM(D) \cong \bK := \PP(V^{*}\otimes \mathfrak{gl}_{2})\git \SL_{2} \times \SL_{2} = \rM_{\PP V}(m^{2}+3m+2) \cong T_{4}$.
\item If $D = aH_{\sigma_{2}} + bT + c\Delta$ for $a, b > 0$ and $c \ge 0$, then $\bM(D) \cong \bX^{1}\git \SL_{2} \times \SL_{2}$, the intermediate space of the partial desingularization of $\bK$.
\item If $D = aH_{\sigma_{2}} + bP + cD_{\deg}$ for $a, b > 0$ and $c \ge 0$, then $\bM(D) \cong \Bl_{\mathrm{OG}(3, \wedge^{2}V)_{\sigma_{(1,1)^{*}}}}\Gr(3, \wedge^{2}V)$. 
\item If $D = aP + bD_{\mathrm{unb}}  + cD_{\deg}$ for $a > 0$ and $b, c \ge 0$, then $\bM(D) \cong \Gr(3, \wedge^{2}V) \cong \Gr(3, \wedge^{2}V^{*})$. 
\item If $D = aH_{\sigma_{1,1}} + bP  + cD_{\mathrm{unb}}$ for $a, b > 0$ and $c \ge 0$, then $\bM(D) \cong \Bl_{\mathrm{OG}(3, \wedge^{2}V)_{\sigma_{(2)^{*}}}}\Gr(3, \wedge^{2}V) \cong \Bl_{\mathrm{OG}(3, \wedge^{2}V^{*})_{\sigma(1,1)^{*}}}\Gr(3, \wedge^{2}V^{*})$. 
\item If $D = aH_{\sigma_{1,1}} + bD_{\mathrm{unb}} + c\Delta$ for $a > 0$ and $b , c \ge 0$, then $\bM(D) \cong \bK^{*} := \PP(V \otimes \mathfrak{gl}_{2})\git \SL_{2} \times \SL_{2} = \rM_{\PP V^{*}}(m^{2}+3m+2)$.
\item If $D = aH_{\sigma_{1,1}} + bT + c\Delta$ for $a, b > 0$ and $c\ge 0$, then $\bM(D) = (\bX^{1}\git \SL_{2} \times \SL_{2})^{*}$, the intermediate space of the partial desingularization of $\bK^{*}$.
\item If $D = aT + b\Delta$ for $a > 0$ and $b \ge 0$, then $\bM(D) \cong \bU := \overline{\mathrm{U}}_{0,0}(\Gr(2, V), 2)$.
\item If $D = aH_{\sigma_{1,1}} + bH_{\sigma_{2}}$ for $a, b > 0$, then $\bM(D) \cong \bC := \mathrm{Chow}_{1,2}(\Gr(2, V))^{\nu}$. 
\item If $D$ is on the boundary of $\Eff(\bM)$, then $\bM(D)$ is a point.
\end{enumerate}
\end{theorem}

\begin{proof}
Items (1), (2), (5), (6), (7), (10), (11) are from \cite[Proposition 3.7, Theorem 3.8, Theorem 3.10]{CC10}. In \cite{CC10}, the range of the divisors giving each model was not stated explicitly. However, since $\bM(D) \cong \bM(D+ E)$ if $E$ is an exceptional divisor of the rational contraction $\bM \dashrightarrow \bM(D)$, it is straightforward to extend the range of divisors. Items (3), (4) are from \cite[Remark 6.7]{CM16}. Items (8), (9) are obtained by the duality map $\Phi$ in Lemma \ref{lem:duality}. Note that for any divisor $D$ on the boundary of $\Eff(\bM)$, $\bM(D)$ is a contraction with positive dimensional fibers of one of normal varieties $\bK$, $\bK^{*}$, and $\Gr(3, \wedge^{2}V)$. Thus it has to be a point since those three varieties have Picard number one.
\end{proof}

\subsection{Relative moduli spaces and its contractions}

When $n > 3$, Mori program for $\bM$ is more complicated. For instance, the movable cone is larger than the case of $n = 3$. To extend Theorem \ref{thm:n=3case}, we need to introduce more birational models of $\bM$. In this section, we introduce new models from the viewpoint of relative moduli spaces. 

The construction of many moduli spaces in Sections \ref{sec:birationalmodels} and \ref{sec:Kronecker} can be relativized. Let $\cS$ be the rank 4 tautological subbundle over $\Gr(4, V^{*})$. Consider the rank 2 Grassmannian bundle $\Gr(2, \cS)$ over $\Gr(4, V^{*})$. 

\begin{definition}\label{def:relativestablemap}
Let $\bM_{\cS} := \overline{\rM}_{0,0}(\Gr(2, \cS), 2)$ be the relative moduli space of stable maps to the Grassmannian bundle. This is a Zariski locally trivial bundle over $\Gr(4, V^{*})$ whose fiber over $S \in \Gr(4, V^{*})$ is $\overline{\rM}_{0,0}(\Gr(2, S), 2)$. 
\end{definition}

There is a functorial morphism 
\[
	r_{M} : \overline{\rM}_{0,0}(\Gr(2, \cS), 2) \to 
	\overline{\rM}_{0,0}(\Gr(2, V^{*}), 2) \cong \bM.
\]
This map is surjective since any degree 2 stable map to $\Gr(2, V^{*})$ factors through $\Gr(2, S)$ for some $S \subset V^{*}$ with $\dim S = 4$. Furthermore, $r$ is not injective precisely on the locus of stable maps whose linear envelope is not 3 dimensional. Thus the exceptional set is the $D_{\deg}$-bundle over $\Gr(4, V^{*})$, which is a divisor. We denote this divisor by $D_{\deg, \cS}$. Then $D_{\deg, \cS}$ is a $\PP^{n-3}$-bundle over $r_{M}(D_{\deg, \cS})$. 

\begin{definition}\label{def:relativesheaves}
Let $\bK_{\cS} := \rM_{\PP \cS}(m^{2}+3m+2)$ be the relative moduli space of semistable sheaves over $\Gr(4, V^{*})$. For each $S \in \Gr(4, V^{*})$, the fiber is $\rM_{\PP S}(m^{2}+3m+2) \cong \PP(S^{*}\otimes \mathfrak{gl}_{2})\git \SL_{2} \times \SL_{2}$. 
\end{definition}

This moduli space can be constructed as an $\SL_{2}\times \SL_{2}$-GIT quotient of the projective space bundle $\PP(\cS^{*}\otimes \mathfrak{gl}_{2})$ over $\Gr(4, V^{*})$. 

There is also a functorial morphism 
\[
	r_{K} : \bK_{\cS} = \rM_{\PP \cS}(m^{2}+3m+2) \to 
	\rM_{\PP V^{*}}(m^{2}+3m+2).
\]
For $S \in \Gr(4, V^{*})$ and $[F] \in \rM_{\PP \cS}(m^{2}+3m+2)|_{S} \cong \rM_{\PP S}(m^{2}+3m+2)$, $r_{K}([F]) = [i_{*}F]$ for $i : \PP S \hookrightarrow \PP V^{*}$. Then $r_{K}$ contracts the locus of $[2\cO_{H}]$ and the fiber of an exceptional point in $\rM_{\PP V^{*}}(m^{2}+3m+2)$ is $\PP^{n-3}$. 

The map $r_{K}$ is not surjective. For instance, $\rM_{\PP V^{*}}(m^{2}+3m+2)$ has an extra component isomorphic to $\Sym^{2}\Gr(3, V^{*})$ which parametrizes $S$-equivalent classes of $\cO_{H}\oplus \cO_{H'}$ for a structure sheaf of a pair of planes. $r_{K}(\bK_{\cS})$ is the closure of the locus of semistable sheaves supported on a smooth quadric surface. 

\begin{definition}\label{def:bL}
Let $\bL$ be the normalization of the image of $r_{K}$ in $\rM_{\PP V^{*}}(m^{2}+3m+2)$. Since $r_{K}$ has connected fibers, $\bL$ is bijective to $r_{K}(\bK_{\cS})$.
\end{definition}

We can relativize the partial desingularization process and obtain a morphism 
\[
	d_{\cS} : \bM_{\cS} \to \bK_{\cS}.
\]
Let $\bX^{1}_{\cS}\git \SL_{2} \times \SL_{2}$ be the intermediate space of the relative partial desingularization.

\begin{lemma}\label{lem:contL}
There is a contraction $\overline{d} : \bM \cong \overline{\rM}_{0,0}(\Gr(2, V^{*}), 2) \to \bL$ which contracts two curve classes $C_{2}$ and $C_{7}$. In particular, $\overline{d}$ contracts $D_{\mathrm{unb}}$ and $\Delta$. 
\end{lemma}

\begin{proof}
\[
	\xymatrix{\overline{\rM}_{0,0}(\Gr(2, \cS), 2) \ar[r]^{d_{\cS}}
	\ar[d]_{r_{M}}&
	\rM_{\PP \cS}(m^{2}+3m+2) \ar[d]^{r_{K}}\\
	\overline{\rM}_{0,0}(\Gr(2, V^{*}), 2) \ar@{-->}[r]^{\overline{d}}&
	\rM_{\PP V^{*}}(m^{2}+3m+2)}
\]
By Lemma \ref{lem:duality}, $D_{\mathrm{unb}}$ on $\bM$ is identified with $D_{\deg}$ on $\overline{\rM}_{0,0}(\Gr(2, V^{*}), 2)$. For a point $f \in D_{\deg} \subset \overline{\rM}_{0,0}(\Gr(2, V^{*}), 2)$, $r_{M}^{-1}(f)$ parametrizes pairs $(f, S)$ where $S \in \Gr(4, V^{*})$ such that the linear envelope of $f$ is a $\PP^{2}$ in $\PP S$. Then $d_{\cS}(r_{M}^{-1}(f))$ parametrizes pairs $([2\cO_{H}], S)$ where $H$ is a plane in $\PP S$. Now $(r_{K}\circ d_{\cS})(r_{M}^{-1}(f))$ is $\{[2\cO_{H}]\}$. By rigidity lemma, there is a morphism $\overline{d} : \overline{\rM}_{0,0}(\Gr(2, V^{*}), 2) \to \rM_{\PP V^{*}}(m^{2}+3m+2)$. Clearly the image has $\bL$ as its normalization. 

From the description of the exceptional set above, it is clear that the curve classes $C_{1}$ and $C_{6}$ on $\overline{\rM}_{0,0}(\Gr(2, V^{*}), 2)$ are contracted. By duality, they correspond to $C_{2}$ and $C_{7}$ on $\bM$. Since deformations of $C_{2}$ (resp. $C_{7}$) cover $D_{\mathrm{unb}}$ (resp. $\Delta$), we obtain the result. 
\end{proof}

The Hilbert scheme construction can also be relativized. 

\begin{definition}\label{def:relativeHilbert}
Let $\bH_{\cS} := \Hilb^{2m+1}(\Gr(2, \cS))$ be the relative Hilbert scheme of conics over $\Gr(4, V^{*})$. This is a Zariski locally trivial bundle over $\Gr(4, V^{*})$ such that for $S \in \Gr(4, V^{*})$, its fiber over $S$ is $\Hilb^{2m+1}(\Gr(2, S))$. 
\end{definition}

\begin{definition}\label{def:B}
Let $\bB := \Bl_{\mathrm{OG}(3, \wedge^{2}\cS)_{\sigma_{(2)^{*}}}}\bG$, the blow-up of the Grassmannian bundle $\bG := \Gr(3, \wedge^{2}\cS)$ along the orthogonal Grassmannian bundle $\mathrm{OG}(3, \wedge^{2}\cS)_{\sigma_{(2)^{*}}}$ which parametrizes $\Sigma_{(2)^{*}}$-planes. 
\end{definition}

We have two birational contractions which are divisorial contractions of $D_{\deg, \cS}$, a divisor parametrizing pairs $(C, S)$ where $C \in \Hilb^{2m+1}(\Gr(2, V^{*}))$ is a conic such that the span $W$ of the union of the spaces parametrized by $C$ is 3-dimensional and $W \subset S \in \Gr(4, V^{*})$.

\begin{equation}\label{eqn:flipHilb}
	\xymatrix{&\bH_{\cS}\ar[ld]_{s} \ar[rd]^{r_{H}}\\
	\bB \ar@{<-->}[rr]^{\phi} && \bH.}
\end{equation}
The map $r_{H}$ is a standard birational morphism 
\[
	r_{H} : \bH_{\cS} \to \Hilb^{2m+1}(\Gr(2, V^{*})) \cong \bH.
\] 
$r_{H}$ sends a pair $(C, S)$ to $C$.

The map $s$ is obtained from the identification 
\[
	\bH_{\cS} \cong 
	\Bl_{2\mathrm{OG}(3, \wedge^{2}\cS)}\Gr(3, \wedge^{2}\cS),
\]
which is the relativization of \eqref{eqn:HilbblowupGr}. Thus $s$ is a blow-down, and it contracts the locus of conics in a $\Sigma_{(1,1)^{*}}$ (as a conic in $\Gr(2, V^{*})$) to a point associated to the plane $\Sigma_{(1,1)^{*}}$.

Recall that there is a regular morphism $env : \Hilb^{2m+1}(\Gr(2, V^{*})) \to \Gr(3, \wedge^{2}V^{*})$ which maps a conic $C \subset \Gr(2, V^{*}) \subset \PP(\wedge^{2}V^{*})$ to a unique $\PP^{2}$ containing $C$ (Definition \ref{def:envelope}). In Section \ref{ssec:determinantalvarieties}, the normalization of the image of $env$ was called $\overline{\bG}$. 

\begin{lemma}\label{lem:flipHilb}
The birational map $\phi$ in \eqref{eqn:flipHilb} is a flip over the blow-up $\widehat{\bG}$ of $\overline{\bG}$ along a subvariety isomorphic to $\mathrm{OG}(3, \wedge^{2}\cS)_{\sigma(2)^{*}}$. 
\end{lemma}

\begin{proof}
Let $\varphi : \bG := \Gr(3, \wedge^{2}\cS) \to \overline{\bG}$ be the morphism obtained from the standard projection $\Gr(3, \wedge^{2}\cS) \to \Gr(3, \wedge^{2}V^{*})$. An element of $\mathrm{OG}(3, \wedge^{2}\cS)_{\sigma_{(2)^{*}}} \subset \Gr(3, \wedge^{2}\cS)$ is a pair $(U, W)$ where $W \in \Gr(4, V^{*})$ and $U \in \Gr(3, \wedge^{2}W)$. Since $U$ is of type $\sigma_{(2)^{*}}$, $U$ is generated by $v_{i}\wedge v$ for $1 \le i \le 3$ and a fixed $v \in W$. Thus $W$, which is the span of $v_{i}$ and $v$, is uniquely determined by $U$. Therefore $\varphi(\mathrm{OG}(3, \wedge^{2}\cS)) \cong \mathrm{OG}(3, \wedge^{2}\cS)$. 

Consider the following diagram:
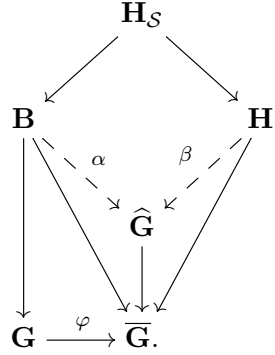
\[
	\xymatrix{
	&\bH_{\cS}\ar[ld] \ar[rd]\\
	\bB \ar[dd] \ar@{-->}[rd]^{\alpha} \ar[rdd] &&
	\bH \ar@{-->}[ld]_{\beta} \ar[ldd]\\
	&\widehat{\bG} \ar[d]\\
	\bG \ar[r]^{\varphi} & \overline{\bG}.}
\]
The preimage of $\mathrm{OG}(3, \wedge^{2}\cS)_{\sigma_{(2)^{*}}}$ in each $\bB$ and $\bH$ is a divisor. From the universal property of blow-up, we obtain two morphisms $\alpha$ and $\beta$. 

From the construction, $\alpha$ and $\beta$ are isomorphisms away from the image of $\mathrm{OG}(3, \wedge^{2}\cS)_{\sigma_{(1,1)^{*}}}$ in $\bH_{\cS}$. A point $x$ in the exceptional set of $\overline{\bG}$ represents a $\Sigma_{(1,1)^{*}} \subset \Gr(2, V^{*})$. Then $\beta^{-1}(x) \cong \PP^{5}$ is the space of conics in the $\Sigma_{(1,1)^{*}}$. $\alpha^{-1}(x)$ the space of $W \in \Gr(4, V^{*})$'s such that $\Sigma_{(1,1)^{*}} \subset \PP(\wedge^{2}W)$. This fiber is isomorphic to $\PP^{n-3}$. 
\end{proof}

\begin{remark}
The blow-up center $\mathrm{OG}(3, \wedge^{2}\cS)$ is isomorphic to $\Gr(3, \cQ)$-bundle over $\PP V^{*}$, where $\cQ$ is the rank $n$ universal quotient bundle over $\PP V^{*}$.
\end{remark}

\subsection{General case}\label{ssec:proof}

Now we are ready to prove the main result. From now, let $n > 3$ and le $V$ be an $(n+1)$-dimensional vector space. In the previous sections, we constructed new birational models of $\bM := \overline{\rM}_{0,0}(\Gr(n-1, V), 2)$. Once the associated model is constructed for each divisor class, as you will see, the proof is very straightforward and there is no technical difficulty since $\bM$ is a Mori dream space (\cite[Corollary 1.2]{CC10}).

\begin{proof}[Proof of Theorem \ref{thm:mainintro}]
Items (1), (2) are \cite[Proposition 3.6]{CC10}. A special case $n = 3$ of Item (3) was proved in \cite[Theorem 3.10]{CC10}, and the same idea can be used to the general case. Item (4) is \cite[Theorem 3.8]{CC10}. 

\textsf{Items (5), (6).} The partial desingularization $d : \bM \to \bK$ contracts two curve classes $C_{1}$ and $C_{6}$, and the second step $c : \bM \to \bX^{1}\git \SL_{2} \times \SL_{2}$ contracts $C_{6}$ (Corollary \ref{cor:partial}). Thus $\bM(aH_{\sigma_{2}}+bT) = \bX^{1}\git \SL_{2} \times \SL_{2}$ if $a, b > 0$ and $\bM(H_{\sigma_{2}}) = \bK$. By Corollary \ref{cor:partial}, $\Delta$ is in the exceptional locus of $\bM \to \bX^{1} \git \SL_{2} \times \SL_{2}$. So $\bM(aH_{\sigma_{2}}+bT + c\Delta) = \bM(aH_{\sigma_{2}}+bT)$. The case of $\bK$ is similar. 

\textsf{Item (7).} Note that $C_{1}$ can be regarded a curve in $\bH$, too. By the contraction $\bH \to \widetilde{\bG}$, $C_{1}$ is contracted (\cite[Proposition 4.13]{HT15}). Since $aH_{\sigma_{2}}+bP$ for $a, b > 0$ are nef divisors on $\bH$ with the property, $\bM(aH_{\sigma_{2}}+bP) = \widetilde{\bE}$. Again, $D_{\deg}$ is in the exceptional locus of $\bM \dashrightarrow \widetilde{\bE}$, $\bM(aH_{\sigma_{2}}+bP + cD_{deg}) = \bM(aH_{\sigma_{2}}+bP)$. 

\textsf{Item (13).} By definition of the divisor class $P$, $\bM(P) = \overline{\bG}$. Since the rational contraction $\bM \dashrightarrow \overline{\bG}$ contracts $D_{\deg}$, $\bM(aP + bD_{\deg}) = \bM(P)$. 

\textsf{Item (14).} The curve class $C_{2}$ can be regarded as a curve in $\bH$. From the proof of Lemma \ref{lem:flipHilb} and the duality, $C_{2}$ is contracted by $\bH \to \widehat{\bG}$. Since $aH_{\sigma_{1, 1}}+bP$ are semiample divisors on $\bH$ contracting $C_{2}$, (14) follows. 

\textsf{Item (12).} By Lemma \ref{lem:contL}, there is a birational contraction $\overline{d} : \bM \to \bL$, which contracts $D_{\mathrm{unb}}$ and $\Delta$. Note that $\overline{d}$ contracts two curve classes $C_{2}$ and $C_{7}$. Thus $\bM(H_{\sigma_{1,1}}) = \bL$. $\Delta$ is in the exceptional locus of $\overline{d}$ so we obtain the result.

\textsf{Item (9).} Let $D = aH_{\sigma_{1,1}}+bP + cD_{\mathrm{unb}}$ for $a, b, c > 0$. Let $\alpha$, $\beta$ be two contractions in the proof of Lemma \ref{lem:flipHilb}. Then by the proof of \textsf{Item (14)}, $-D$ is $\beta$-ample. Let $B$ be the 1-parameter family of data $(\Sigma_{(1,1)^{*}}, W)$ such that $\Sigma_{(1,1)^{*}} \subset \PP(\wedge^{2}W)$ and $W$ forms a line in $\Gr(4, V^{*})$. Then $B$ is contracted by $\alpha$. Furthermore, the curve cone of the exceptional set is generated by $B$, since it is isomorphic to a projective space. Thus if $D \cdot B_{1} > 0$, then $D$ is $\alpha$-ample and the proof is completed. Note that $aH_{\sigma_{1,1}}+bP$ is the pull-back of an ample divisor on $\widehat{\bG}$ by \textsf{Item (14)}. Thus $B \cdot H_{\sigma_{1,1}} = B \cdot P = 0$. Also since $W$ varies, the push-forward of $B$ by $\bB \to \bG \to \Gr(4, V^{*})$ is a curve. Now $D_{\mathrm{unb}}$ is the pull-back of an ample divisor on $\Gr(4, V^{*}) \cong \Gr(n-3, V)$ by Definition \ref{def:divisorclasses}, $B \cdot D_{\mathrm{unb}} > 0$.

\textsf{Item (8).} Since the contraction $\bM(aH_{\sigma_{1,1}}+bP + cD_{\mathrm{unb}}) = \bB \to \overline{\bG} = \bM(P)$ factors through $\bG = \Gr(3, \wedge^{2}\cS)$, either $\bM(aP + bD_{\mathrm{unb}})$ or $\bM(aH_{\sigma_{1,1}}+bD_{\mathrm{unb}})$ for $a, b > 0$ is $\bG$. It is straightfoward to check that the push-forward of the curve class $C_{1}$ is contracted by $\bB \to \bG$. Thus $\bM(aP + bD_{\mathrm{unb}}) \cong \bG$. $D_{\deg}$ is an exceptional divisor of $\bB \to \bG$. Thus we can obtain the result.

\textsf{Item (10).} The relative contraction $\bH_{\cS} \to \bK_{\cS}$ descends to $\bB \to \bK_{\cS} \cong \rM_{\PP \cS}(m^{2}+3m+2)$ by rigidity lemma. $\bK$ admits two morphisms to $\rM_{\PP V^{*}}(m^{2}+3m+2)$ and $\Gr(4, V^{*})$. Thus $H_{\sigma_{1,1}}$ and $D_{\mathrm{unb}}$ are two semiample divisors on $\rM_{\PP \cS}(m^{2}+3m+2)$. The product morphism 
\[
	\bK = \rM_{\PP \cS}(m^{2}+3m+2) \to 
	\rM_{\PP V^{*}}(m^{2}+3m+2) \times \Gr(4, V^{*})
\]
is injective. Thus $aH_{\sigma_{1,1}} + bD_{\mathrm{unb}}$ with $a, b > 0$ is an ample divisor on $\rM_{\PP \cS}(m^{2}+3m+2)$. Therefore $\bM(aH_{\sigma_{1,1}}+bD_{\mathrm{unb}}) \cong \bK$. Finally, $\Delta$ is an exceptional divisor for $\bM \dashrightarrow \rM_{\PP \cS}(m^{2}+3m+2)$, so we obtain the statement.

\textsf{Item (11).} Let $\bR$ be the normalization of the image of the product map 
\[
	\bM \to \rM_{\PP V^{*}}(m^{2}+3m+2) \times \bU.
\]
The first map is obtained from Lemma \ref{lem:contL}. Now it is clear that $\bR$ has two birational morphisms to $\bL$ and $\bU$. Furthermore, since both $\bM \to \bL$ and $\bM \to \bU$ contracts $C_{7}$, so is $\bM \to \bR$. Thus $\bM(aH_{\sigma_{1,1}}+bT) \cong \bR$. Since $\Delta$ is contracted by $\bM \to \bR$, $\bM(aH_{\sigma_{1,1}}+bT+c\Delta) \cong \bR$. 

\textsf{Item (15).} Since $\bM$ is a Mori dream space, $\bM(a\Delta + bD_{\deg})$ is a contraction of $\bM(aH_{\sigma_{2}}+bD_{\deg}+c\Delta) \cong \bK$. Since the Picard number of $\bK$ is one, $\bM(a\Delta + bD_{\deg})$ has to be a point. 

\textsf{Item (16).} From Definition \ref{def:divisorclasses}, we obtain $\bM(D_{\mathrm{unb}}) = \Gr(4, V^{*})$. Since $\Delta$ is in the exceptional locus of the rational contraction $\bM \dashrightarrow \Gr(4, V^{*})$, $\bM(aD_{\mathrm{unb}}+b\Delta) \cong \bM(D_{\mathrm{unb}}) = \Gr(4, V^{*})$. The other case is similar.
\end{proof}

\section{Applications}\label{sec:application}

A quick application of describing birational morphisms between models in terms of explicit contractions is the computation of topological invariants. In this section, we leave two computations of motivic invariants of double symmetroid $T_{4} \cong \bK$ and the moduli space $\rM_{\PP^{2}}(4m+2)$ of semistable torsion sheaves on $\PP^{2}$. 

\subsection{Motivic invariants of the double symmetorid}\label{ssec:symmetroid}

An explicit description of the partial desingularization in Section \ref{ssec:partialdesing} enables us to compute the virtual Poincar\'e polynomial of $T_{4} \cong \bK$. A nice summary of the definition and basic properties of the virtual Poincar\'e polynomial $P(X)$ of a projective variety $X$ can be found in \cite[Section 2]{Mun08}. The virtual Poincar\'e polynomial of $\bM$ were calculated by A. Mart\'in by using Bia\/lynicki-Birula decomposition:

\begin{proposition}[\protect{\cite[Theorem 3.1]{LM14}}]
The virtual Poincar\'e polynomial of $\bM$ is
\[
	\frac{[(1+q^{n+1})(1+q^3)-q(1+q)(q^2+q^{n-1})]
	(1-q^{n+1})(1-q^{n})(1-q^{n-1})}{(1-q)^3(1-q^2)^2}.
\]
\end{proposition}

\begin{proposition}\label{prop:virtualdouble}
The virtual Poincar\'e polynomial of $T_4$ is
\begin{equation}\label{eqn:PT4}
	P(\bM) - (P(\overline{\rM}_{0,0}(\PP^{n-1}, 2))-1)
	\left(\frac{1-q^{n}}{1-q}\right)
	- \left(P((\PP^{n-2})^{2})-1\right)
	\left(\frac{1}{2}\left(P(\PP^{n})^{2}+\frac{1-q^{2n+2}}{1-q^{2}}
	\right)-P(\PP^{n})\right).
\end{equation}
\end{proposition}

\begin{proof}
Let $d : \bM \to T_4$ be the desingularization morphism in Theorem \ref{thm:Kirwandes}. On $T_4$, let $\rN_{i}$ be the locally closed subvariety parametrizing rank $i$ quadrics. By Corollary \ref{cor:partial}, for any closed point $x$ on $\rN_{1} \cong \PP V^{*}$, $d^{-1}(x) \cong \overline{\rM}_{0,0}(\PP^{n-1}, 2)$ and for any closed point $y \in \rN_{2} \cong (\PP {V^{*}}\times\PP {V^{*}}-\Delta)/\ZZ_{2}$ where $\Delta$ is the diagonal of $\PP {V^{*}}\times\PP {V^{*}}$, $d^{-1}(y) \cong (\PP^{n-2})^{2}$. Thus we obtain
\[
	P(T_4) = P(\bM)
	- P(\rN_{1})P(\overline{\rM}_{0,0}(\PP^{n-1}, 2)) + P(\rN_{1})
	- P(\rN_{2})P((\PP^{n-2})^{2}) + P(\rN_{2})\\
\]
By \cite[Section 2]{Mun08}, $P(\rN_{2})=\frac{1}{2} (P(\PP^n)^2+\frac{1-q^{2n+2}}{1-q^2})-P(\PP^n)$. Thus we obtain \eqref{eqn:PT4}.
\end{proof}

Note that $P(\overline{\rM}_{0,0}(\PP^{n-1}, 2))=\frac{(1-q^{n+1})(1-q^n)(1-q^{n-1})}{(1-q)^2(1-q^2)}$ by \cite[Theorem 1.3]{KM10}.

\begin{remark}
Let $T_{4}(n)$ be the double symmetroid for the ($n+1)$-dimensional vector space $V$. By using computer algebra system, we are able to obtain a simpler expression of $P(T_{4}(n))$ for small $n$. For instance, 
\begin{enumerate}
\item $P(T_4(3))=(q^7+q^6+q^2+q+1)(q^2+1)$,
\item $P(T_4(4))=(q^7-q^6+q^5-q^4+q^2-q+1)(q^4+q^3+q^2+q+1)(q^2+q+1)$,
\item $P(T_4(5))=(q^{13}+q^{12}+q^{11}+q^{10}-q^8-q^7+q^4+q^3+q^2+q+1)(q^4+q^2+1)$, and
\item $P(T_{4}(6)) = (q^{15} + q^{13} + q^{11} - q^{10} - q^8 + q^4 + q^2 + 1)(q^6 + q^5 + q^4 + q^3 + q^2 + q + 1)$.
\end{enumerate}
\end{remark}

\subsection{Motivic invariants of the pure sheaves on $\PP^2$}\label{ssec:Poinpolymodulisheaves}

Let $\rM_{\PP^{2}}(dm+\chi)$ be the moduli space of one-dimensional semistable sheaves on $\PP^{2}$ with Hilbert polynomial $dm+\chi$. In several  papers including \cite{CC12, CM14, Yua14, CC15, Ien16}, people have computed $P(\rM_{\PP^{2}}(dm+\chi))$ when $d$ and $\chi$ are coprime. If $d$ and $\chi$ are not coprime, because of the existence of singular locus, the computation of the Poincar\'e polynomial seems to be hard. Proposition \ref{prop:virtualdouble} and Bridgeland wall crossing (\cite[Section 6.1]{BMW14}) enables us to compute $P(\rM_{\PP^{2}}(4m+2))$, which is the first non-trivial case with $(d, \chi) \ne 1$. 

\begin{proposition}\label{prop:Ppolysheaves}
The virtual Poincar\'e polynomial of $\rM_{\PP^{2}}(4m+2)$ is given by
\[
\begin{split}
P(\rM_{\PP^{2}}(4m+2)) &=\; q^{17}+2q^{16}+5 q^{15}+9 q^{14}+11 q^{13}+11 q^{12}+10 q^{11}+10 q^{10}\\
&+9 q^9+10 q^8+10 q^7+12 q^6+12 q^5+12 q^4+9 q^3+5 q^2+2 q+1.
\end{split}
\]
\end{proposition}
\begin{proof}
From \cite[Section 4]{DM11}, we know that $T_4(5)$ is birational to $\rM_{\PP^{2}}(4m+2)$. Furthermore, as described in \cite[Section 6.1]{BMW14}, there are two wall-crossings from $\rM_{\PP^{2}}(4m+2)$ to $T_4(5)$. An object $F$ in the exceptional locus in each wall-crossing is an extension of a particular type, described in Table \ref{tbl:wall}. After the wall-crossing, we obtain new extensions $F'$. Here $\vee$ denotes the derived dual $\mathcal{RH}om(-,\cO_{\PP^{2}})$. Now by a simple calculation, one can see that
\[
P(\rM_{\PP^{2}}(4m+2))= (P(\PP^{14})-P(\PP^2))+ P(\PP^2\times \PP^2)(P(\PP^{12})-1)+P(T_4(5)).
\]
Combining these with Proposition \ref{prop:virtualdouble}, we obtain the result.
\end{proof}

\begin{table}[!ht]
\begin{tabular}{|c|c|}
\hline
The first wall&The second wall\\
\hline\hline
$\ses{\cO_{\PP^2}(1)}{F}{\cO_{\PP^{2}}(-3)[1]}$ & $\ses{I_p(1)}{F}{I_q^{\vee} (-3)[1]}$ for $p$ and $q\in \PP^2$\\
\hline
$\ses{\cO_{\PP^{2}}(-3)[1]}{F'}{\cO_{\PP^2}(1)}$ &$\ses{I_q^{\vee} (-3)[1]}{F'}{I_p(1)}$ for $p$ and $q\in \PP^2$ \\
\hline
\end{tabular}
\medskip
\caption{Bridgeland wall-crossings from $\rM_{\PP^{2}}(4m+2)$}
\label{tbl:wall}
\end{table}



\end{document}